\documentclass[11pt,letterpaper,reqno]{amsart}

\usepackage{amsmath,amscd}

\usepackage{setspace}
\usepackage{multicol}
\usepackage[lite]{amsrefs} %used to create the bibliography:
\usepackage{amssymb} %special symbols in the AMSfonts package, listed in http://www.ctan.org/tex-archive/info/symbols/math/symbols.pdf
\usepackage{graphicx} %the updated version of the older graphics.sty
\usepackage{pdflscape}
\usepackage{euscript}
\usepackage{color}
\usepackage{array}
\usepackage{picture}
\usepackage{epic}
\usepackage{tikz}
 \usetikzlibrary{backgrounds,cd}
\usepackage{comment}
\usepackage{enumerate}
\usepackage{hyperref}
\usepackage[margin=1.25in]{geometry}
\usepackage{caption}
\usepackage{MnSymbol}
\usepackage{enumitem}
\usepackage{fancyhdr} %no headers on every page, page number place in upper right
\usepackage{centernot}
\usepackage{mathtools}
\usepackage{stmaryrd}
\usepackage{float}
\usepackage{ytableau}
\newcommand{\Spec}{\operatorname{Spec}}
\newcommand{\discrep}{\operatorname{discrep}}
\newcommand{\mult}{\operatorname{mult}}
\newcommand{\Stab}{\operatorname{Stab}}
\newcommand{\Ind}{\operatorname{Ind}}
\newcommand{\reg}{\operatorname{reg}}

\newcommand{\fa}{{\mathfrak a}}             % more Fraktur
\newcommand{\fb}{{\mathfrak b}}

\newcommand{\fg}{{\mathfrak g}}
\newcommand{\fh}{{\mathfrak h}}

\newcommand{\fl}{{\mathfrak l}}

\newcommand{\fn}{{\mathfrak n}}

\newcommand{\fp}{{\mathfrak p}}

\newcommand{\fu}{{\mathfrak u}}
\newcommand{\fs}{{\mathfrak s}}

\newcommand{\mcO}{\mathcal O}
\newcommand{\mcM}{\mathcal M}

\newcommand{\ga}{\alpha}

\newcommand{\ol}[1]{\overline{{#1}}}

\numberwithin{equation}{section}

\theoremstyle{plain} %% This is the default, anyway
\newtheorem{theorem}{Theorem}[section]
\newtheorem{corollary}[theorem]{Corollary}
\newtheorem{lemma}[theorem]{Lemma}
\newtheorem{proposition}[theorem]{Proposition}
\newtheorem*{proposition*}{Proposition}

\newtheorem{conjecture}[theorem]{Conjecture}

\theoremstyle{definition}
\newtheorem{definition}[theorem]{Definition}

\newtheorem{example}[theorem]{Example}

\theoremstyle{remark}
\newtheorem{remark}[theorem]{Remark}

\newcommand{\C}{\mathbb{C}}

\newcommand{\V}{\mathcal{V}}   
   
\newcommand{\A}{\mathbb{A}}
\newcommand{\Q}{\mathbb{Q}}   
\newcommand{\Z}{\mathbb{Z}}   

\newcommand{\co}{\mathcal{O}}

\usepackage{amsmath,amscd}

\makeatletter
\@namedef{subjclassname@2020}{%
  \textup{2020} Mathematics Subject Classification}
\makeatother

\hypersetup{bookmarksdepth=2}

\begin{document}

\title[Nilpotent orbit covers for $SL_n$]{Functions on nilpotent orbit covers and birational geometry}

\author{William Graham}
\address{Department of Mathematics, University of Georgia, Athens GA 30602, USA}
\email{wag@uga.edu}
\author{Scott Joseph Larson}
\address{Department of Mathematics and Statistics, College of Charleston, Charleston SC 29424, USA}
\email{larsonsj@cofc.edu}
\author{Alberto San Miguel Malaney}
\address{School of Mathematics and Statistics, University of Glasgow, Glasgow, G12 8QQ,
UK}
\email{alberto.malaney@glasgow.ac.uk}

\date{\today}
\thanks{This material is based upon work supported by the National Science Foundation under Grants No. DMS-2302116 and DMS-2302117.}

\begin{abstract}
We use an analogue of the Springer resolution to describe the $G$-module structure on the ring of regular
functions on the universal cover $\widetilde{\co}$ of any nilpotent orbit for $G = SL_n$. Building on previous work on
the extended Springer resolution, we construct a variety $\widetilde{\mcM}$ that is finite over the cotangent bundle of a partial flag variety $G/P$, and proper and birational over the affinization $\mcM$ of $\widetilde{\co}$. We use techniques in birational geometry to show that $\widetilde{\mcM}$ has rational singularities, which provides the  cohomology vanishing needed to describe the ring of functions on $\widetilde{\co}$ as an induced representation from a Levi subgroup of $G$.  Our results also yield a description of the structure of
$R(\widetilde{\co})$ as a graded $G$-module.  We describe the minimal embedding of $\mcM$, study the lifting of characters of the component group of $\widetilde{\co}$ to parabolics and Levi subgroups,
and make a more general vanishing conjecture.
\end{abstract}

\maketitle 

\thispagestyle{empty}
\setcounter{tocdepth}{1}
%\tableofcontents %to include a table of contents
%\linespread{1.2}
\parskip=4pt \baselineskip=14pt

\section{Introduction}
Let $G$ be a complex simply connected semisimple group with Lie algebra $\fg$.  Let $\co \subset \fg$ be a nilpotent $G$-orbit with universal cover $\widetilde{\co}$,
and let $R(\co)$ and $R(\widetilde{\co})$ denote their rings of regular functions.  These and related rings have long been of interest
in representation theory (see for example \cite{LMBM}, \cite{Vog90}).  

The study of the $G$-module structure of $R(\co)$ was initiated by Kostant \cite{Kos},
who proved that if $\co$ is the principal (that is, regular) nilpotent orbit, then as a $G$-module, $R(\co) = \Ind_H^G (0)$, where $H$ is a maximal torus of $G$, and $0$ corresponds to the trivial
character of $H$.  For an arbitrary nilpotent orbit $\co$, McGovern \cite{McG} obtained a formula for $R(\co)$ as a linear combination of terms of the
form $\Ind_L^G (V_i)$.  Here $L$ is a Levi subgroup of a Jacobson-Morozov parabolic subgroup of $G$ constructed out of an $\mathfrak{sl}_2$-triple containing an element of $\co$,
and the $V_i$ are explicitly described representations of $L$.  McGovern's general formula is an equality on the level of virtual representations, since in general some of the
coefficients in his formula are negative.  For $SL_n$, the situation simplifies, since every nilpotent orbit is strongly Richardson, and McGovern observed (see \cite[p.~215]{McG}) that if
one chooses $L$ differently---more precisely, if $L$ is a Levi subgroup of a Richardson parabolic for $\co$---then
$R(\co) = \Ind_L^G(0)$.

An analogous formula for general nilpotent orbit covers has not appeared.  Although 
McGovern proved the existence of a formula for $R(\widetilde{\co})$ as a sum of virtual
representations, of a form similar to his general formula for $R(\co)$, there has been no explicit
formula except for certain orbits.  For the regular orbit, a formula for $R(\widetilde{\co})$ is
given in \cite{Gra1992}.   Both \cite{McG} and \cite{Gra1992} rely essentially on the vanishing of certain cohomology
groups, proved in \cite{McG} using a resolution of singularities of $\Spec R(\co)$ and 
the Grauert-Riemenschneider theorem, and in \cite{Gra1992} by
using a theorem of Hesselink \cite{He76}.  Sommers conjectured a formula for the graded $G$-module decomposition of $R(\widetilde{\co})$ for any $G$
(see Conjecture 4.2 of \cite{Som}).  Sommers stated that he could prove a variation of this conjecture in type $A$ which would yield a 
formula for $R(\widetilde{\co})$, but a proof has not appeared.

The monograph \cite{LMBM} studies orbit covers using a different method involving Harish-Chandra
bimodules.  Formulas for $R(\mcM)$ for birationally rigid orbit covers, for linear classical
groups, can be obtained directly from results in that
book (see \cite{LMBM}, Proposition 3.10 (iii), Proposition 6.36, and Section 10.5).  For $G = SL_n$, these are the universal covers of orbits corresponding to partitions where
all parts are equal.  The $G$-module structure on functions on orbit covers is compatible
with birational induction, as can be seen from \cite[Prop.~7.13]{LMBM}, or more directly
from Namikawa's paper \cite{Na1}.   Combining the formulas for birationally rigid orbit covers
from \cite{LMBM} with the explicit description of induction of orbit covers in \cite{Na1}
provides a different approach to obtaining the (non-graded) formula for $R(\mcM)$ in our paper.

In this paper we obtain formulas for the $G$-module decomposition of $R(\widetilde{\co})$ for any nilpotent orbit $\co$,
in the case where $G = SL_n$.  For even orbits, these coincide with the formulas conjectured by Sommers.  
We also study the graded structure of $R(\widetilde{\co})$ as a $G$-representation and as
an $R(\co)$-module.   We
describe the minimal embedding of $\mcM = \Spec R(\widetilde{\co})$
in the sense of Brylinski and Kostant \cite{BrylinskiKostant1994}
and prove that the map $\mcM \to \overline{\co}$ is a bijection over the boundary of $\co$, extending
\cite[Remark 4.15]{BrylinskiKostant1994} to arbitrary orbits.
As in the formula for $R(\co)$ for $SL_n$ given in \cite{McG},
we make use of the fact that for $SL_n$, any $\co$ is strongly Richardson,
i.e., there is a resolution of singularities $T^*(G/P) \to \overline{\co}$ of the orbit closure $\co$, for some Richardson parabolic
subgroup $P$ of $G$.  

We now describe our results in more detail.  
Suppose $\co = \co_p$ is the orbit
corresponding to a partition $p = (p_1, \ldots, p_{t})$ of $n$.  Let $p'$ denote the transpose partition
of $p$, and let $L$ denote the subgroup of block diagonal matrices with block sizes given from left to right
by the parts of $p'$.  The Richardson parabolics for $\co$ are the parabolic subgroups
with a Levi factor conjugate to $L$.  Let $P = LU$ be the parabolic subgroup containing the Borel subgroup
$B$ of lower triangular matrices.  We modify the resolution $T^*(G/P) \to \overline{\co}$ to obtain
a map $\widetilde{\mcM} \longrightarrow \mcM := \Spec R(\widetilde{\co})$.
We construct $\widetilde{\mcM}$ explicitly and prove that although $\widetilde{\mcM}$ is not smooth in general, it
has rational singularities.  Using this, we show that as $G$-modules,
\begin{equation} \label{e:main}
R(\widetilde{\co}) = \sum_{j=0}^{k-1} \Ind_L^G(\mu_j)
\end{equation}
(see Theorem \ref{t:rings}), where $k$ is the order of the fundamental group of $\co$.
The $\mu_j$ are explicitly determined weights which give characters $e^{\mu_j}$ of $L$,
and the characters of the component group are obtained from the restrictions of the $e^{\mu_j}$.

This choice of parabolic is well-suited for applying results from birational geometry,
but Theorem \ref{t:rings} can be used to obtain other expressions for $R(\widetilde{\co})$
using different weights and different subgroups.  In particular, we
index the $\mu_j$
so that $\mu_0 = \lambda_0 = 0$ and for $j >0$, $\mu_j$ is conjugate under the Weyl group to the fundamental dominant
weight $\lambda_{j n /k}$ (in the notation of \cite{Hum}).    Each part of $p'$ occurs with multiplicity
divisible by $k$.  Let $d$ be the partition with the same parts as $p'$, but each with multiplicity $1/k$ of its
multiplicity in $p'$, and let $q$ be the sequence formed by concatenating $k$ copies of $d$.
Let $M$ be the Levi subgroup of block diagonal matrices with block sizes given from
left to right by the entries of $q$.  Then $M$ is conjugate to $L$, and
\begin{equation} \label{e:main2}
R(\widetilde{\co}) = \sum_{j=0}^{k-1} \Ind_M^G(\lambda_{j n /k}).
\end{equation}

Observe that all the terms in \eqref{e:main} and \eqref{e:main2} are positive, so there is no cancellation.
Also, while the Levi subgroup $M$ in \eqref{e:main2} 
depends on the orbit,
the weights $\lambda_{j n/k}$ depend only on the order $k$ of the fundamental
group of the orbit.  In fact, the formulas for both $R(\co)$ and $R(\widetilde{\co})$ in $SL_n$
are strongly reminiscent of the corresponding formulas in the case where $\co$ is principal.

Because induction is compatible with conjugation by elements of $G$, and we can simultaneously
conjugate the Levi subgroup $L$ and the weight $\mu_j$ to $M$ and $\lambda_{j n/k}$, 
\eqref{e:main2} is a straightforward consequence of \eqref{e:main}.
Moreover, for even orbits,
these Levi subgroups and weights are conjugate to those considered in
\cite{Som}, and these formulas agree with the formulas that would follow
from \cite[Conjecture 4.2]{Som}. 
However, the ring $R(\mcM)$
is graded, and to study this grading, we must consider not only the Levi subgroups $L$ and $M$, but the
parabolic subgroups $P = LU$ and $Q = M U_Q$.  In particular, we must show that the characters
$e^{\mu_j}$ and $e^{\lambda_{j n/k}}$ lift characters of the component group $A(\co)$ of the orbit
to $P$ and $Q$, respectively.  For $P$, this follows from an explicit description of component groups.
The parabolic subgroups $P$ and $Q$ are not $G$-conjugate, but their Levi subgroups
are, and the results of Section \ref{s:lifting} show that the lifting property for $Q$ follows from the
lifting property for $P$ and the $G$-conjugacy of the pairs $(L, \mu_j)$ and
$(M, \lambda_{j n/k})$.

The results in this paper permit a description of
$R(\widetilde{\co})$ as a graded $G$-module in terms of variants of
Lusztig's $q$-analogues of weight multiplicity formulas (see \cite{Pan}).
However, the $q$-analogs that appear in the graded
formulas depend on the unipotent radical of the parabolic subgroup, not only on the
Levi factors.  If $p$ has parts of different sizes, then $P$ and $Q$ are not conjugate, so they yield different
formulas for the graded $G$-module structure.  For the same reason, the
formulas we prove are different from the formulas for the graded $G$-module structure than would
follow from \cite[Conjecture 4.2]{Som}. 

Because the weights $\lambda_{jn/k}$ are dominant, the expressions using these
weights and the parabolic $Q = M U_Q$ are well-suited
to study the grading on $R(\widetilde{\co})$.  Combining our work with
a result of Grantcharov \cite{Grant}, we show that $R(\widetilde{\co})$ is generated as an $R(\co)$-module
by the unique copies of the minuscule representation it contains, along with the trivial representation.
We describe the minimal embedding
of $\mcM$ in the sense of Brylinski and Kostant, and show that the natural map $\mcM \to \overline{\co}$, which is a $k$-fold
covering over the orbit $\co$, is bijective over the boundary of $\co$, a result due
to Brylinski and Kostant for the principal orbit (see \cite[Remark 4.13]{BrylinskiKostant1994}).

The cohomology vanishing needed to obtain \eqref{e:main} would follow from \cite{Bro} if the
weights $\mu_j$ were dominant, but they are not, and it is essential to our proof
that we have cohomology vanishing for these nondominant weights.  Even though the
weights in \eqref{e:main2} are dominant, we need \eqref{e:main} to prove this result.
The weights $\mu_j$ emerge naturally from the explicit construction of $\widetilde{\mcM}$,
and the necessary vanishing is a consequence of the statement that $\widetilde{\mcM}$ has rational singularities.
We can prove this using methods from birational geometry because we have an explicit construction
of $\widetilde{\mcM}$.  This circumvents the algebraic difficulties encountered in trying to generalize
Hesselink's result (see Remarks \ref{r:vanish} and \ref{r:vanish2} for further discussion).

The construction of $\widetilde{\mcM}$ resembles the construction in the last section of \cite{Gra2022} for arbitrary
nilpotent orbits in all types.  However, the two constructions differ, except for orbits parametrized by partitions
for which all parts are equal.  Indeed, the construction in \cite{Gra2022} is built out of the resolution used in \cite{McG},
which uses a Jacobson-Morozov parabolic associated to an element of $\co$, while the construction
here uses the resolution $T^*(G/P) \to \overline{\co}$.  Nevertheless, the
construction here does make use of the toric results used for the regular nilpotent orbit in \cite{Gra2022}.
It seems likely to us that the ideas we use here can be adapted to the construction for other orbits in \cite{Gra2022} at
least for $SL_n$.   We expect that this would yield formulas for $R(\widetilde{\co})$ 
consistent with Sommers's conjecture.  We also hope to investigate extensions of the ideas
in this paper to nilpotent orbits in other types.

Finally, $\widetilde{\mcM}$ can be defined in the more general setting of symplectic singularities.
It would be interesting to explore explicit descriptions of this variety in other examples.

An outline of the paper is as follows.  Section \ref{GenPic} gives a general construction that
is used in the proof of our main results.  Section \ref{s:Det} shows that if $M$ is an affine space of block diagonal square matrices
and $Y$ is obtained from $M$ by adjoining a $k$-th root of the determinant function, then $Y$ has rational singularities.
Section \ref{s:Springer} constructs the extended Springer resolution $\widetilde{\mcM} \to \mcM$.  Although $\widetilde{\mcM}$ is
not smooth, we deduce from the results of the previous section the key fact that $\widetilde{\mcM}$ has rational singularities.
In Section \ref{s:Gmod}, we deduce the formulas \eqref{e:main} and \eqref{e:main2}  for the $G$-module decomposition of $R(\widetilde{\mcM})$. 

In order to study the graded structure of $R(\mcM)$, we need to show that lifting characters of component groups
to parabolic subgroups is compatible with conjugation of Levi subgroups.  The necessary results are given in
Section \ref{s:lifting}, which also contains a description of minimal lifts as equally distributed characters.
Section \ref{s:grading} contains results
related to the grading on $R(\mcM)$ which can be used to obtain a description of the graded $G$-module structure.
Section \ref{s:vanishing} contains a more general cohomology vanishing conjecture and some consequences,
and explains the relation of the formulations in this paper with
Sommers's conjecture. 

\medskip

{\bf Acknowledgments:} We thank Eric Sommers for helpful feedback, which in particular led
to the inclusion of the graded versions of the results on
the $G$-module structure. We thank Valery Alexeev for a useful discussion about birational geometry,
and Nikolay Grantcharov for sharing his preprint \cite{Grant} and for providing a counterexample
to an earlier version of Conjecture \ref{conj:vanishing}.

\medskip

{\bf AI Disclosure:} We used ChatGPT, Claude, and Gemini for background research, to find references and for 
computing examples. 
Additionally, some of the techniques used to show $Y$ has canonical singularities in the proof of Theorem \ref{RRat} were suggested by ChatGPT. All writing 
was done by the authors, and we maintain full responsibility for the paper.

\vfill

\pagebreak

\subsection{Conventions and notation} Throughout the paper we work with schemes over $\C$.
We write $R(X)$ for the ring of regular functions on a scheme $X$.
Let $G = SL_n$ with diagonal maximal torus $H$ and center $Z \cong \Z_n$, and let $B$ be the Borel subgroup consisting of lower triangular matrices.
A parabolic subgroup $P$ is called standard if it contains $B$. A Levi subgroup $L$ of $G$ is called standard if it is the Levi factor of a standard parabolic $P$ such that $H\subset L$.
The Weyl group of $G$ is denoted by $W$.  The letter $T$ will denote the diagonal maximal torus in $SL_k$, where
$k$ will be an integer dividing $n$. By a reductive part of a linear algebraic group we mean a reductive complement to the unipotent radical; i.e., a Levi factor.
We use this term mostly for stabilizer groups, and reserve the term Levi subgroup for a Levi factor of a parabolic subgroup.
For any $P$-module $V$, let $H^i(G/P, V)$ denote the $i$-th cohomology
group of the sheaf of sections of $G \times^P V \to G/P$.

As usual, we denote the Lie algebra of an algebraic group by the corresponding fraktur letter, so 
$\fg$, $\fh$, and $\fb$ are the Lie algebras of $G$, $H$, and $B$.  Given a root $\ga \in \fh^*$ of $G$, we write
$\fg_{\ga}$ for the corresponding root subspace of $\fg$.  
We choose the positive system such that the root spaces $\fg_{\ga} \subset \fb$ correspond to negative roots,
so $B$ is the ``negative" Borel subgroup.  We identify $\fh^*$ with $\C^n / \C \cdot (1, \ldots, 1)$, so
the fundamental dominant weights are $\lambda_i = (1^i, 0^{n-i})$, where the exponential notation
denotes repeated indices.  The minuscule weights are the Weyl group conjugates of the $\lambda_i$.

Given a dominant weight $\mu$ for $G$, $V_{\mu}$ denotes the irreducible $G$-module with highest
weight $\mu$.  
Given a Levi subgroup $L$, we write $e^{\lambda}: L \to \C^*$ for a homomorphism whose derivative is $\lambda \in \fl^*$.
If $L$ contains $H$, then $\lambda$ is
determined by its restriction to $\fh$, so we can identify $\lambda$ with 
an element of $\fh^*$.  The character $e^{\lambda}$ must be trivial on the semisimple subgroup $[L, L]$, 
so it is determined by its values on the center $Z(L)$, and $\lambda$ (viewed
as an element of $\fh^*$) is orthogonal to the coroots of $L$.  As usual $\Ind_L^G(\lambda)$ is
the subspace of $R(G)$ consisting of functions $f \in R(G)$ such that $f(g \ell) = e^{\lambda}(\ell^{-1}) f(g)$ for all $g \in G, \ell \in L$.
If $P = LU$ is a parabolic subgroup with Levi subgroup $L$ and unipotent radical $U$, 
we extend a character $e^{\lambda}$ of $L$ to a character of $P$
which is trivial on $U$.  

The nilpotent orbits in $\fg$ are parametrized by partitions $p$ of
$n$, with the orbit $\co_p$ consisting of matrices of Jordan type $p$.  If $p$ is understood, we usually
simply write $\co$ for the orbit.  Let $G^e$ denote the stabilizer in $G$ of $e \in \co$
with identity component $G^e_0$.  Then $\co \cong G/G^e$, with universal cover $\widetilde{\co} = G/G^e_0$.
The component group of $\co$ is $A(\co) = G^e/G^e_0$. 
It does not depend on a choice of $e$ because for $e, e' \in \co$, $G^e/G^e_0$ and $G^{e'}/G^{e'}_0$ are canonically identified via conjugation since $A(\co)$ is abelian.
Since $G$ is simply connected, $A(\co) \cong \pi_1(\co, e)$. 

We frequently write matrices whose only nonzero entries are in certain blocks (such as diagonal or block diagonal matrices) simply
by writing the sequence of blocks.

\section{Relative Normalizations} \label{GenPic}

Recall from \cite[Definition 4.1.24]{Liu} that the normalization $g: Z \to Y$ of an integral scheme $Y$ in a field extension $L$ of $K(Y)$ is uniquely characterized by the properties that $Z$ is a normal variety with $K(Z) = L$, and that $g$ is an integral morphism that extends the natural map $\Spec(L) \to Y$. For the remainder of the section, consider a variety $X$ along with a birational map $\pi: Y \to X$ and a finite surjective map $f: \mcM \to X$. We define $g: \widetilde{\mcM} \to Y$ to be the normalization of $Y$ in the function field $K(\mcM)$ of $\mcM$.  

\begin{lemma} \label{RelNormMap}

There is a map $\varphi: \widetilde{\mcM} \to \mcM$ providing a commutative diagram
% https://q.uiver.app/#q=WzAsNCxbMiwyLCJYIl0sWzAsMiwiXFxtY00iXSxbMiwwLCJZIl0sWzAsMCwiXFx3aWRldGlsZGV7XFxtY019Il0sWzEsMCwiZiJdLFsyLDAsIlxccGkiLDJdLFszLDIsImciXSxbMywxLCJcXHZhcnBoaSIsMl1d
\[\begin{tikzcd}
	{\widetilde{\mcM}} && Y \\
	\\
	\mcM && X
	\arrow["g", from=1-1, to=1-3]
	\arrow["\varphi"', from=1-1, to=3-1]
	\arrow["\pi"', from=1-3, to=3-3]
	\arrow["f", from=3-1, to=3-3]
\end{tikzcd}\]
\end{lemma}

\begin{proof}

The normalization in a field extension is a local construction (see \cite[p.~121]{Liu}). Thus, we take affine open subsets $\Spec(A) \subset X$ and $\Spec(R) \subset Y$ such that $\Spec(R)$ maps to $\Spec(A)$, and we let $\Spec(B)$ be the preimage of $\Spec(A)$ under $f$. Note that $\Spec(B)$ is finite over $\Spec(A)$, so every element of $B$ is integral over $A$. Furthermore, since $\pi$ is a birational map between integral schemes, the map $A \to R$ is injective. Hence every element of $B$ is integral over $R$, so $B$ is contained in the integral closure $R'$ of $R$ in $K(B) = K(\mcM)$. Thus, we obtain a canonical map $\Spec(R') \to \Spec(B)$, and $\varphi: \widetilde{\mcM} \to \mcM$ is given by patching these together. The commutativity of the diagram is automatic since all of the ring maps are given by inclusions in $K(\mcM)$. 
\end{proof}

Note that $\widetilde{\mcM}$ is normal, $g$ is finite (see \cite[Proposition 4.1.27]{Liu}), and $\varphi$ is birational. Additionally, by the universal property of the fiber product we have a canonical map 
$$
h: \widetilde{\mcM} \to \mcM \times_X Y.
$$ 
Let $n: N \to \mcM \times_X Y$ be the normalization of $\mcM \times_X Y$. Since $\widetilde{\mcM}$ is normal, the universal property of the normalization (see \cite[Exercise II.3.8]{Ha}) implies that the map $h$ above lifts to a map $j: \widetilde{\mcM} \to N$.

\begin{lemma} \label{jIso}

The map $j$ is an isomorphism.

\end{lemma}

\begin{proof}

We have a commutative diagram
% https://q.uiver.app/#q=WzAsNixbNSw2LCJcXG1jTSJdLFs1LDAsIlkiXSxbMywzLCJcXG1jTSBcXHRpbWVzX1ggWSJdLFs3LDMsIlgiXSxbMiwzLCJOIl0sWzAsMywiXFx3aWRldGlsZGV7XFxtY019Il0sWzIsMSwicF8yIl0sWzEsMywiXFxwaSJdLFswLDMsImYiXSxbMiwwLCJwXzEiXSxbNCwyLCJuIl0sWzUsNCwiaiJdLFs1LDEsImciXSxbNSwwLCJcXHZhcnBoaSJdXQ==
\begin{equation*}
\begin{tikzcd}
	&&&&& Y && \\
	\\
	\\
	{\widetilde{\mcM}} && N & {\mcM \times_X Y} &&&& X \\
	\\
	\\
	&&&&& \mcM
	\arrow["\pi", from=1-6, to=4-8]
	\arrow["g", from=4-1, to=1-6]
	\arrow["j", from=4-1, to=4-3]
	\arrow["\varphi", from=4-1, to=7-6]
	\arrow["n", from=4-3, to=4-4]
	\arrow["{p_2}", from=4-4, to=1-6]
	\arrow["{p_1}", from=4-4, to=7-6]
	\arrow["f", from=7-6, to=4-8]
\end{tikzcd}
\end{equation*}

Since $g$ is finite, $j$ must be finite as well. Since $\pi$ is birational, by base change $p_1$ must be birational as well. Since $\varphi$, $n$, and $p_1$ are all birational, $j$ must be birational as well. Thus it is a finite birational map between normal varieties, so it is an isomorphism.  
\end{proof}

\begin{remark}

We briefly describe one case of particular interest. Let $X$ be a conical affine symplectic singularity. It follows from \cite[Theorem 1]{Bra} that $\pi_1(X_{\reg})$ is finite. Thus, we can let $R$ be the ring of functions of the universal cover of $X_{\reg}$, and $\mcM := \Spec(R) \to X$ be the finite map induced by $R(X) \cong R(X_{\reg}) \hookrightarrow R$. Furthermore, let $\pi: Y \to X$ be a crepant (partial) resolution. A natural question is whether the resulting space $\widetilde{\mcM}$ has rational singularities. We show in Section \ref{s:Springer} that this is the case when $X$ is a nilpotent orbit closure in $\fs \fl_n$ and $\pi: T^*(G/P) \to X$ is a particular symplectic resolution.  

\end{remark}

\section{Rational singularities and roots of determinants} \label{s:Det}

In this section we use birational geometry to study the singularities of a variety $Y$ that can be obtained by adding a $k$th root of the determinant to an affine space $X$ given by block diagonal matrices. This will be essential in the next section because we will construct a space $\widetilde{\mcM}$ whose geometry will be closely related to the singularities of $Y$. 

Let $n_1, \dots, n_{\ell}$ be positive integers, and for each $i$, let $M_i \cong \A^{n_i^2}$ be the affine space of $n_i$-by-$n_i$ matrices. 
Let $X = M_1 \times \dots \times M_{\ell}$.  Let $f_i \in R(M_i)$ be the determinant function, and let $f =  f_1 f_2 \dots f_{\ell} \in R(X)$. Finally, let $R = R(X)[z]/(z^k - f)$ and $Y = \Spec(R)$.

\begin{theorem} \label{RRat}

$Y$ has rational singularities. 

\end{theorem}

\begin{proof}

We are going to show that $Y$ has canonical singularities (see \cite[Definition 2.34]{KM}). Note that $R$ is a local complete intersection ring, which implies that $Y$ is Gorenstein and thus that $Y$ is Cohen-Macaulay and $K_{Y}$ is a Cartier divisor. Moreover, $Y$ is $R_1$ because the singular locus of $Y$ is given by the simultaneous vanishing of $f$ and the differential $Df$, which is a locus of codimension at least 2.  Hence, $Y$ is normal since it is $R_1$ and Cohen-Macaulay. Thus by \cite[Corollary 5.24]{KM}, if $R$ has canonical singularities, it also has rational singularities. Since $K_{Y}$ is Cartier, to show that $Y$ has canonical singularities we only need to show that $\discrep(Y) \geq 0$.

Let $\pi: Y \to X$ be given by the inclusion $R(X) \hookrightarrow R$. This map is a cyclic cover with ramification divisor given by $D := V(f)$. Let $E \subset Y$ be the divisor lying over $D$. Because $\pi$ is an order $k$ cyclic cover we have 
$$
\pi^*D = kE \hspace{10pt} \text{and} \hspace{10pt} K_{Y} = \pi^* K_{X} + (k - 1)E,
$$
where the second equation follows from the ramification formula (see e.g~\cite[Proposition 20.2]{Kol}). In particular, the $\Q$-divisor $\Delta_X := \frac{k - 1}{k} D$ satisfies $\pi^*(K_{X} + \Delta_X) = K_{Y}$. Thus, by \cite[Proposition 5.20(3)]{KM}, we have 
\begin{equation} \label{discrep>}
\discrep(Y) \geq \discrep(X, \Delta_X).
\end{equation}

Since $K_Y$ is Cartier, $\discrep(Y)$ is an integer, and thus it is sufficient to show that the pair $(X, \Delta_X)$ is klt. Furthermore, since $0 \leq \frac{k - 1}{k} < 1$, we can check that $(X, \Delta_X)$ is klt by checking that for a fixed log resolution $\pi_X: \tilde{X} \to X$, all the discrepancies of the exceptional divisors are greater than -1 (see \cite[Theorem 2.44, Proposition 2.41]{KM}). In fact, we shall see that for the log resolution described below, all of the discrepancies turn out to be greater than 1. Given a log resolution $g: \tilde{Z} \to Z$ of a pair $(Z, \Delta)$ and an irreducible exceptional divisor $E \subset \tilde{Z}$, we denote by $a(E, Z, \Delta)$ the discrepancy of $E$ with respect to $(Z, \Delta)$.

We will use a log resolution of $(X, \Delta_X)$ that is given by taking a product of log resolutions of each $(M_i, V(f_i))$ that we describe below. These log resolutions and their relevant properties are summarized succinctly in \cite[Theorem 2.5]{KMN}; for proofs or further details see \cite[Chapter 4]{Joh} or \cite{Vai}. The log resolution, $\pi_i: \tilde{M}_i \to M_i$, is given as an iterated blowup for $j = 0, \dots n_i - 2$, where $A_0 := M_i$ and $b_{j + 1}: A_{j + 1} \to A_j$ is the blowup of $A_j$ at the strict transform $Z_j$ of the closed subvariety given by matrices of rank at most $j$. It is shown in the proof of \cite[Theorem 4.4]{Joh} that each $Z_j$ is smooth. Thus, each of these blowups has a single exceptional divisor, whose strict transform in $\tilde{M}_i := A_{n_i - 1}$ we call $E_{ij}$. It is automatic that 
$$
\pi_X := \pi_1 \times \dots \times \pi_{\ell}: \tilde{X} := \tilde{M}_1 \times \dots \times \tilde{M}_{\ell} \to M_1 \times \dots \times M_{\ell}
$$
is a log resolution of $(X, \Delta_X)$ and that each exceptional divisor is of the form 
$$
E^i_{j} = \tilde{M}_1 \times \dots \times \tilde{M}_{i - 1} \times E_{ij} \times \tilde{M}_{i + 1} \times \dots \times \tilde{M}_{\ell}.
$$
Finally $\Delta_X = \Delta_1 + \dots + \Delta_{\ell}$ where $\Delta_i := \frac{k - 1}{k} V(f_i)$, and for each $E^i_{j}$ we have 
$$
a(E^i_{j}, X, \Delta_X) = a(E_{ij}, M_i, \Delta_i).
$$

In order to compute the discrepancy at each $E_{ij}$ in the log resolution $\pi_i: \tilde{M}_i \to M_i$ of $(M_i, \Delta_i)$, we apply \cite[Lemma 2.29]{KM} inductively to each blow up. In order to simplify the notation, for this computation we drop the $i$ subscripts. In particular, $M$ consists of $n$-by-$n$ matrices, and $\Delta = \frac{k - 1}{k} D$, where $f$ is the determinant function and $D = V(f)$. We denote by $D_j$ the strict transform of $D$ in $A_j$, and by $\Delta^j = \frac{k - 1}{k} D_j$ the strict transform of $\Delta$ in $A_j$. We denote by $E_j$ the exceptional divisor of $b_{j + 1}: A_{j + 1} \to A_j$. Following the notation in the proof of \cite[Theorem 4.4]{Joh}, we use the same notation for the strict transform of $E_j$ in all the $A_k$, for $k > j + 1$. 

Let $j \in \{0, \dots, n - 1\}$. Since for $j \geq 1$, $\{E_0, \dots, E_{j - 1}\}$ are the exceptional divisors of the composition of blowups $\rho_j: A_j \to M$, we have the equation 
\begin{equation} \label{Indj}
K_{A_j} + \Delta^j = \rho_j^* (K_M + \Delta) + \sum_{k = 0}^{j - 1} a_k E_k.
\end{equation}
We shall prove by induction that for each $j \in \{0, \dots, n - 1\}$, every $a_k$ in \eqref{Indj} is greater than 1. The base case of $j = 0$ is automatic. For the inductive step, fix some $j \in \{0, \dots, n - 1\}$ and assume \eqref{Indj} holds with $a_0, \dots, a_{j - 1}$ all greater than 1. Then we can apply \cite[Lemma 2.29]{KM} to the blowup $A_{j + 1}$ of $Z_j$ to obtain the equation 
\begin{equation} \label{Blj}
K_{A_{j + 1}} + \Delta^{j + 1} - \sum_{k = 0}^{j - 1} a_k E_k = b_{j + 1}^*(K_{A_j} + \Delta^j - \sum_{k = 0}^{j - 1} a_k E_k) + a_j E_j
\end{equation}
where 
$$
a_j := a(E_j, A_j, \Delta^j  - \sum_{k = 0}^{j - 1} a_k E_k) = c - 1 - \tfrac{k - 1}{k} \mult_{Z_j} D_j  + \sum_{k = 0}^{j - 1} a_k \mult_{Z_j} E_k.
$$
with $c$ the codimension of $Z_j$ in $A_j$. Note that $c = (n - j)^2$, and since the generic point of $Z_j$ is not contained in $E_k$ for $k < j$, we have $\mult_{Z_j} E_k = 0$. Additionally, it follows from \cite[Corollary 4.5]{Joh} that $\mult_{Z_j} D_j = n - j$. Combining these facts, we obtain the equation
\begin{align}
a_j &= (n - j)^2 - 1 - \tfrac{k - 1}{k}(n - j) \\
&= (n - j)(n - j - \tfrac{k - 1}{k}) - 1 > 2*1 - 1 = 1.
\end{align}
Combining the inductive hypothesis \eqref{Indj} with the equation \eqref{Blj} we obtain 
\begin{equation} \label{Indj+1}
K_{A_{j + 1}} + \Delta^{j + 1} = b_{j + 1}^*(K_{A_j} + \Delta^j - \sum_{k = 0}^{j - 1} a_k E_k) + \sum_{k = 0}^{j} a_k E_k = \rho_{j + 1}^* (K_M + \Delta) + \sum_{k = 0}^{j} a_k E_k,
\end{equation}
where all the $a_j$ are greater than 1. Thus our induction is complete, and for $j = n - 1$ we obtain
$$
K_{\tilde{M}} + \tilde{\Delta} = \pi_X^* (K_M + \Delta) + \sum_{k = 0}^{n - 2} a_k E_k
$$
with $\tilde{\Delta} := \Delta^{n - 2}$ and with all the $a_j$ are greater than 1. Thus we have shown that all the discrepancies of $\pi_X: Y \to X$ are greater than 1. As explained above, this implies that $(X, \Delta_X)$ is klt, and thus by \eqref{discrep>}, $Y$ has canonical and thus rational singularities. 
\end{proof}

\begin{remark}

Note that despite the fact that all the discrepancies of exceptional divisors of $\pi_X: \tilde{X} \to X$ are strictly greater than 1, the contributions of the strict transform of $\Delta_X$ can cause $\discrep(X, \Delta_X)$ to be negative. However, in the case where $l = 1$ (i.e. $R$ is given by adjoining a root of a single determinant function), $D = V(f)$ is a prime divisor and thus its strict transform in $\tilde{X}$ is smooth. Therefore, it follows from \cite[Corollary 2.32(2)]{KM} that $\discrep(X, \Delta_X) = \frac1k$, which implies that $(X, \Delta_X)$ and $Y$ have terminal singularities.

\end{remark}

\section{Construction of the Extended Springer Resolution} \label{s:Springer}

Let $P \subset G$ be a parabolic subgroup of a connected complex reductive algebraic group, with Lie algebra $\fp \subset \fg$. There is a unique nilpotent orbit $\mcO_P$ whose 
intersection with the nilradical $\fu$ of $\fp$ is dense in $\fu$. This orbit is called the Richardson orbit of $P$, and a nilpotent orbit is Richardson if it is the Richardson orbit of some parabolic subgroup. 
We will say that $P$ is a Richardson parabolic for $\co$ (the term polarization is often used for such $P$). 
There is a natural Springer map 
\begin{equation} \label{PSpr}
T^*(G/P) \cong G \times^P \fu \to \overline{\mcO_P},
\end{equation}
which depends only on the conjugacy class of $P$. In general, this map is proper, and finite over $\mcO_P$. The degree of the map over $\mcO_P$ is sometimes called the Hesselink number of the (conjugacy class of the) parabolic $P$, because these numbers were studied and computed in classical type in \cite{He}. A nilpotent orbit is ``strongly Richardson'' if it is the Richardson orbit of a parabolic with Hesselink number one, i.e., if its closure is resolved by a Springer map \eqref{PSpr}.  It is shown in \cite{Fu} that every symplectic resolution of (the normalization of) a nilpotent orbit closure is of the form \eqref{PSpr}. 

In type $A$ the situation simplifies significantly because every nilpotent orbit is strongly Richardson. Thus, let $G = SL_n$ and let $\mcO_{p} \subset \fs \fl_{n}$ be the nilpotent orbit with Jordan blocks of size $p = (p_1, \dots, p_t)$, where \(p_i\geq p_{i+1}\). Additionally, let $p'$ be the conjugate partition to $p$. 
Let $B$ denote the lower triangular ``negative" Borel subgroup of $G$ and \(H\subset B\) the diagonal Cartan subgroup.
Let $P \supset B$ be a standard parabolic subgroup of $G$ with  Levi subgroup $L\supset H$.
It is a fact that $\mcO_{p}$ is the Richardson orbit of $P$ if and only if the sizes of the
blocks of $L$, read from left to right, are a permutation of the entries of $p'$ (see \cite[Theorem 3.3]{He}).  

For the remainder of this section, fix $P\supset B$ and its Levi part $L\supset H$ so that the sizes of the Levi blocks are non-increasing from the upper left to the lower right in $SL_n$:
in other words they are the entries of $p'$, in order. Let $U$ be the unipotent radical of $P$, so $P=LU$ and $\fp = \fl + \fu$.

Let $k$ be the greatest common divisor of the parts of the partition $p$, so by definition, $k$ divides the multiplicity of each entry in $p'$. By definition, $p_1$ is the number of entries in $p'$, i.e., the number of blocks of $L$. Thus, we have 
$$
L \cong S \left( \prod_{i = 1}^{p_1} GL_{p'_i} \right) .
$$
Let $V_p \subset \fu$ be the vector space generated by root spaces $\fg_{\alpha}$ such that when $\alpha$ is written as a sum of negative simple roots, at most one of them is not in $L$. Note that $\fu$ is a representation of $L$ under the adjoint action, and as an $L$-module, it decomposes as $\fu \cong V_p \oplus [\fu,\fu]$. Hence \(V_p\cong\fu/[\fu,\fu]\) as \(L\)-representations.

Explicitly, $L$ is the subgroup of block diagonal matrices of $G$ consisting of  $p_1$ blocks, where the $i$-th block is square of size $p'_{i}$.  The space $V_p$ consists of the matrices whose nonzero entries are in the $p_1 - 1$ blocks immediately below the diagonal blocks.  Thus, $V_p \cong M_1 \times \cdots \times  M_{p_1 - 1}$, where $M_i$ is the space of $p'_{i + 1}$-by-$p'_{i}$ matrices. In particular, $M_i$ consists of square matrices whenever $i \not\equiv 0$ mod $k$. We can write the adjoint action of \(L\) on \(V_p\) in terms of the block decompositions described above as 
$$
(A_1, A_2, \dots, A_{p_1}) \cdot (B_1, B_2, \dots, B_{p_1 - 1}) = (A_2 B_1 A_1^{-1}, \dots, A_{p_1} B_{p_1 - 1} A_{p_1 - 1}^{-1}).
$$

Let $e \in V_p$ consist of the matrices such that each block $B_i$ of $e$ has entries $(B_i)_{rs} = \delta_{rs}$.  In other words, each $B_i$ is either the identity, or it consists of a left block that is the identity and a right block that is zero. Note that the Jordan form of $e$ has partition type $p$, and thus $e \in \mcO_p$. Additionally,  $L \cdot e \subset V_p$ is the set of matrices such that each $B_i$ has maximal rank.  Indeed, the equation 
$$
(A_1, A_2, \dots, A_{p_1}) \cdot e = (B_1, B_2, \dots, B_{p_1 - 1})
$$
can be solved whenever all the $B_i$ have maximal rank: let $A_{p_1}$ be any invertible matrix and let the other $A_i$ be chosen inductively such that 
$A_{i + 1} [\delta_{rs}] = B_i A_i$.

We shall describe the stabilizer $L^e$ below. It is a straightforward computation that 
$$
A = (A_1, \dots A_{p_1}) \in L^e
$$
is equivalent to $A_i = A_{i + 1}$ if $p'_i = p'_{i + 1}$ and $A_i$ has block form
\begin{equation} \label{e:blocks1}
A_i = 
\begin{bmatrix}
    A_{i + 1}   & 0 \\
    B   & C \\
\end{bmatrix}
\end{equation}
if $p'_i > p'_{i + 1}$. This implies that $A_1$ determines all the other $A_i$. We can write $A_1$ in the form 
\begin{equation} \label{e:blocks2}
A_1 = 
\begin{bmatrix}
    C_1       & 0 & 0 & 0 & \dots & 0 \\
    B_{21}       & C_{2} & 0 & 0 & \dots & 0 \\
   B_{31}   & B_{32}    & C_{3} & 0 & \dots & 0 \\
    \hdotsfor{6} \\
    B_{\ell 1}   & B_{\ell 2} & B_{\ell 3} & \dots & B_{\ell - 1, \ell} & C_{\ell}
\end{bmatrix}
\end{equation}
where $\ell$ is the number of distinct parts of the partition $p$, and $C_i$ is a square matrix of size given by the multiplicity $m_i$ of the $i$th largest distinct entry of $p$. The only other condition on $A$ is that its determinant is one, which does not impose any additional conditions on the $B_{ij}$ matrices above. Note the condition that $\det(A) = \det(A_1 A_2 \dots A_{p_1}) = 1$ is precisely 
$$
\det(C_1^{u_1} C_2^{u_2} \dots C_{\ell}^{u_{\ell}} ) = \det(C_1)^{u_1} \det(C_2)^{u_2}  \dots \det(C_{\ell})^{u_{\ell}}  = 1
$$
where $u_1 \dots, u_{\ell}$ are the ${\ell}$ distinct entries in $p$. In summary,
\begin{equation} \label{Le}
L^e \cong \A^m \times \{C_i \in M_{m_i \times m_i} | \det(C_1)^{u_1} \det(C_2)^{u_2}  \dots \det(C_l)^{u_{\ell}}  = 1 \}.
\end{equation}
This is a disconnected group with component group $\Z/k \Z$ and identity component 
\begin{equation} \label{Le0}
L^e_0 \cong \A^m \times \{C_i \in M_{m_i \times m_i} | \det(C_1)^{\frac{u_1}{k}} \det(C_2)^{\frac{u_2}{k}}  \dots \det(C_{\ell})^{\frac{u_{\ell}}{k}}  = 1 \}.
\end{equation}
(The equation on the right hand side defines a connected group, since it is a fiber of $\beta \circ \alpha$, where
$\alpha (C_1, \ldots, C_{\ell}) =  (\det C_1, \ldots, \det C_{\ell})$ and $\beta(z_1, \ldots, z_\ell) = z^{\frac{u_1}{k}} \dots z{\frac{u_{\ell}}{k}}$.  Since
the fibers of the determinant map are connected, so are the fibers of $\alpha$, and the fibers of $\beta$ are isomorphic to $\C^*$, as the $u_i/k$ have
greatest common divisor $1$.)
We shall now explicitly construct an $L$-equivariant cover $X_p$ of $V_p$ such that stabilizer of a lift $e_p$ of $e$ is $L^e_0$.

Let $T$ be the diagonal maximal torus in $SL_k$, and choose the positive system of roots of $SL_k$ such that the positive root spaces are upper triangular matrices. Let $\V_{ad} \subset \fs \fl_{k}$ be the subspace given by the $-\ga_i$ root spaces, where $\ga_1, \ldots, \ga_{k-1}$ are the simple positive roots. This construction comes from \cite[Section 3.2]{Gra2022}, 
although in this paper the sign convention differs because $\V_{ad}$ consists of the $-\ga_i$ root spaces instead of the $\ga_i$ root spaces.  Note that $T$ acts on $\V_{ad}$ by the adjoint 
action, and we can identify $R(\V_{ad})$ with the subring $\C[x_1, \ldots, x_{n-1}]$
of $R(T)$, where $x_i = e^{-\alpha_i}$.  

We define a group homomorphism $\rho: L \to T$ by 
\begin{equation} \label{rhoDef}
(A_1, \dots A_{p_1}) \mapsto \left( \prod_{i \equiv 1 \hspace{-5pt} \mod k} \det A_i, \prod_{i \equiv 2 \hspace{-5pt} \mod k} \det A_i, \dots, \prod_{i \equiv k \hspace{-5pt} \mod k} \det A_i \right).
\end{equation}
Similarly, we define $f: V_p \to \V_{ad}$ by 
\begin{equation} \label{fDef}
(B_1, \dots B_{p_1 - 1}) \mapsto \left( \prod_{i \equiv 1 \hspace{-5pt} \mod k} \det B_i, \prod_{i \equiv 2 \hspace{-5pt} \mod k} \det B_i, \dots, \prod_{i \equiv k - 1 \hspace{-5pt} \mod k} \det B_i \right).
\end{equation}
By construction, these maps intertwine the adjoint action of $L$ on $V_p$ with the adjoint action of $T$ on $\V_{ad}$. In other words, if we let $L$ act on $\V_{ad}$ via $\rho$, then $f$ is $L$-equivariant. 

In \cite[Section 3]{Gra2022}, a finite $T$-equivariant cover $\pi: \V \to \V_{ad}$ is constructed, such that $\V$ is a toric variety for $T$. Hence $L$ acts on $\V$ via $\rho$, and thus $L$ acts on the fiber product
$$ 
X_p := V_p \times_{\V_{ad}} \V
$$ 
componentwise, ensuring that the projection maps $\pi_1: X_p \to V_p$ and $\pi_2: X_p \to \V$ are $L$-equivariant. 

\begin{lemma} \label{teStab}

The stabilizer of any $e_p \in \pi_1^{-1}(e) \subset X_p$ is $L_0^e$.

\end{lemma}

\begin{proof}

Since $\V$ is constructed as a toric variety for $T$, we have $T \subset \V$. Thus, any point in the open $T$-orbit of $\V$ has trivial stabilizer in $T$, so its stabilizer in $L$ is $\ker(\rho)$. Note that $e_p = (e, v)$ for some $v \in \pi^{-1}(1, \dots, 1)$. Thus, $v$ is in the open $T$-orbit of $\V$ and we have 
$$
L^{e_p} = L^e \cap \ker(\rho).
$$
By Definition \eqref{rhoDef}, $\ker(\rho)$ is given by 
$$
\{(A_1, \dots A_{p_1}) \in L \  | \ \prod_{i \equiv j} \det A_i = 1 \text{ for } j = 1, \dots, k\}.
$$
If  $(A_1, \dots A_{p_1})  \in L^e$, then in the notation of \eqref{Le}, for each $j = 1, \dots, k$ we have 
\begin{equation} \label{e:prod}
\prod_{i \equiv j} \det A_i = \det(C_1)^{\frac{u_1}{k}} \det(C_2)^{\frac{u_2}{k}}  \dots \det(C_l)^{\frac{u_{\ell}}{k}}.
\end{equation}
Thus by \eqref{Le0} we are done. 
\end{proof}

\begin{remark} \label{r:center}
The center $Z_k$ of $SL_k$ is isomorphic to $\Z_k$, and consists of the matrices of the form $\zeta I$, where $\zeta$ is a $k$-th
root of unity.  The homomorphism $\rho$ takes the stabilizer $L^e$ to $Z_k$; precisely, $\rho(A_1, \ldots, A_{p_1}) = \zeta I$, where
$\zeta$ is given by the right hand side of \eqref{e:prod}.  This induces an isomorphism $L^e/L^e_0 \cong Z_k$.
\end{remark}

\begin{example} \label{example:SL(12)-1}

Consider the partition $p = (6,  3^2)$ and the orbit $\mcO_p \subset \fs \fl_{12}$. The conjugate partition to $p$ is $p' = (3^3, 1^3)$, so $L$ is the Levi 
subgroup of $G$ with 3 blocks of size 3 in the upper left and 3 blocks of size 1 in the lower right.  Additionally, $V_p$ consists of two $3 \times 3$ matrices, one $3 \times 1$ matrix, and two 
$1 \times 1$ matrices. The stabilizer $L^e$ consists of all elements of $L$ of the form $\ell = (A, A, A, a, a, a)$ such that 
$$
A = 
\begin{bmatrix}
    a   & 0 \\
    B   & C \\
\end{bmatrix}
$$
where $B$ is a $1 \times 2$ matrix and $C$ is a $2 \times 2$ matrix. The condition that $\det \ell = 1$ imposes the condition that $a^6 \det(C)^3 = 1$. Thus, we have 
$$
L^e \cong \{(a, B, C) \in GL_1 \times \A^2 \times GL_2 \ | \ a^6 \det(C)^3 = 1\}.
$$
The identity component $L^e_0$ is given by matrices of the same form, but subject to the stronger condition that $a^2 \det(C) = 1$. Thus, we have
$$
L^e_0 \cong \{(a, B, C) \in GL_1 \times \A^2 \times GL_2 \ | \ a^2 \det(C) = 1\}.
$$
Finally, the greatest common divisor $k$ of $p$ is $3$, and $\rho: L \to T$ and $f: V_p \to \V_{ad}$ are given respectively by
$$
(A_1, A_2, A_3, a_4, a_5, a_6) \mapsto (\det(A_1) a_4, \det(A_2) a_5, \det(A_3) a_6) 
$$
$$
(B_1, B_2, B_3, b_4, b_5) \mapsto (\det(B_1) b_4, \det(B_2) b_5).
$$
Using the notation from above, we see that for $\ell= (A, A, A, a, a, a) \in L^e$, 
we have $\rho(\ell) = (a^2 \det(C), a^2 \det(C), a^2 \det(C))$, and thus $\ker(\rho) \cap L^e = L^e_0$. 

\end{example}

In order to study the singularities of $X_p$ we shall study a finite cover of it. Following \cite[Proposition 3.5]{Gra2022}, there is a finite quotient map $\V_k \to \V$, where $\V_k = \Spec(\C[w_1, \dots, w_{k - 1}])$, and the composition $\V_k \to \V_{ad}$ corresponds to the map $R(\V_{ad}) \to R(\V)$ defined by $x_i \mapsto w_i^k$. Let $Y_p := V_p \times_{\V_{ad}} \V_k$ so that we have the following diagram, where both squares are Cartesian:
% https://q.uiver.app/#q=WzAsNixbMCwwLCJZX3AiXSxbMiwwLCJYX3AiXSxbNCwwLCJWX3AiXSxbNCwyLCJcXFZfe2FkfSJdLFsyLDIsIlxcViJdLFswLDIsIlxcVl9rIl0sWzAsMSwicSJdLFsxLDIsIlxccGlfMSJdLFsyLDMsImYiLDJdLFs1LDRdLFs0LDMsIlxccGkiXSxbMSw0LCJcXHBpXzIiLDJdLFswLDVdXQ==
\begin{equation} \label{YXVp}
\begin{tikzcd}
	{Y_p} && {X_p} && {V_p} \\
	\\
	{\V_k} && \V && {\V_{ad}}
	\arrow["q", from=1-1, to=1-3]
	\arrow[from=1-1, to=3-1]
	\arrow["{\pi_1}", from=1-3, to=1-5]
	\arrow["{\pi_2}"', from=1-3, to=3-3]
	\arrow["f"', from=1-5, to=3-5]
	\arrow[from=3-1, to=3-3]
	\arrow["\pi", from=3-3, to=3-5]
\end{tikzcd}
\end{equation}

\begin{corollary} \label{YRat}

$Y_p$ has rational singularities. 

\end{corollary}

\begin{proof} 

By definition $Y_p$ is an affine scheme whose ring of functions is 
$$
R \cong R(V_p) [z_1, \dots, z_{k - 1}]/(w_1^k - f_1, \dots, w_{k - 1}^k - f_{k - 1}),
$$
where $f_j$ is the product of determinants of the components $B_i$ of $V_p$ with $i \equiv j \mod k$. Thus, $Y_p$ is simply the product of an affine space with varieties of the form in Theorem \ref{RRat}. Thus, by Theorem \ref{RRat}, we are done.
\end{proof}

\begin{corollary} \label{XRat}

$X_p$ has rational singularities. 

\end{corollary}

\begin{proof} 

By Corollary \ref{YRat}, $Y_p$ is normal. Since \eqref{YXVp} is Cartesian and $\V_k \to \V$ is a quotient by a finite group, $q: Y_p \to X_p$ is as well. Thus, by \cite[Proposition 6.4.1]{BH}, $X_p$ is normal. Since $q$ is finite, $X_p$ is normal, and $Y_p$ has rational singularities, \cite[Proposition 5.13]{KM} implies that $X_p$ has rational singularities as well. 
\end{proof}

\begin{corollary} \label{XInt}

$\pi_1: X_p \to V_p$ is the normalization of $V_p$ in $K(L/L_0^e)$.

\end{corollary}

\begin{proof}
We first show that $X_p$ is irreducible.  First, $X_p$ is normal since it has rational singularities by Corollary \ref{XRat}.
Since $L/L^e$ is a dense open subset of $V_p$, the preimage $N := \pi_1^{-1}(L/L^e)$ is a dense open subset of $X_p$. By Lemma \ref{teStab}, $N$ contains $L \cdot e_p \cong L/L^{e_p} \cong L/L^e_0$.  It follows from \cite[Section 3]{Gra2022} that $\pi: \V \to \V_{ad}$ is a degree $k$ finite map, and that it is \'etale over the complement of the coordinate hyperplanes given by $x_i = 0$. This implies that $\pi_1: X_p \to V_p$ is also finite of degree $k$, and that $N$ is \'etale over $L/L^e \subset V_p$. Thus, $N$ and $L \cdot e_p$ are both degree $k$ \'etale covers of $L/L^e$, and $L \cdot e_p \subset N$, so $N = L \cdot e_p \cong L/L^e_0$. Since $N \cong L/L^e_0$ is an irreducible dense open subset of $X_p$, $X_p$ is irreducible.  It is also reduced since it is normal, so $X_p$ is integral and $K(X_p) = K(L/L^e_0)$. 

Since $\pi_1: X_p \to V_p$ is finite, $X_p$ is normal, and $K(X_p) = K(L/L^e_0)$, $X_p$ is indeed the normalization of $V_p$ in $K(L/L_0^e)$.
\end{proof}

Returning to the global picture, let $\widetilde{\mcO}_p$ be the universal cover of $\mcO_p$. We have natural $G$-equivariant maps
\begin{equation} \label{univCover}
\mcM := \Spec R(\widetilde{\mcO}_p) \to \Spec R(\mcO_p) \to \ol{\mcO}_p,
\end{equation}
where the first map is induced by the inclusion $R(\mcO_p) \hookrightarrow R(\widetilde{\mcO}_p)$ and the second is the normalization. It follows from \cite{KP} that 
the second map is an isomorphism as we are in type $A$ (outside of type $A$ this is not always true).  We can also view $\mcM$ as the normalization of $\ol{\mcO}_p$ in $K(\widetilde{\mcO}_p) = K(\mcM)$, and this description  extends to all types (see \cite{Na1}).  Since normalization is a local construction, this implies that $\widetilde{\mcO}_p$ is the preimage of $\mcO_p$ under \eqref{univCover}. Additionally, $\mcM$ and the analogous spaces in other types have symplectic singularities (see \cite[Lemma 2.5]{Lo21}), and were studied from this perspective in \cite{Na1} 
\cite{Na2}, \cite{Mat}.

The natural projection $\fu \to \fu/[\fu,\fu] \cong V_p$ is $P$-equivariant  if we inflate the $L$-action on $V_p$ to a $P$ action by letting $U$ act trivially. We can also inflate the action of $L$ on $X_p$ to a $P$-action, so $\pi_1: X_p \to V_p$ is $P$-equivariant. Thus, we have a $P$-variety
$$
\tilde{\fu} := X_p \times_{V_p} \fu.
$$
Observe that as a space, $\tilde{\fu} = X_p \times [\fu, \fu]$, but the fiber product definition is used to endow $\tilde{\fu}$ with a $P$-action.
Since $X_p$ is irreducible, so is $\tilde{\fu}$.  Let $\tilde{e} = (e_p, 1) \in \tilde{\fu}$.
We have a $G$-equivariant map
$$
\tilde{\eta}: \widetilde{\mcM} := G \times^P \tilde{\fu} \to G \times^P \fu \cong T^*(G/P).
$$

\begin{corollary} \label{l:Mrational}
$\widetilde{\mcM}$ has rational singularities.
\end{corollary}

\begin{proof}
The map $\widetilde{\mcM} = G \times^P \widetilde{\fu} \to G/P$ is a locally trivial fiber bundle with fibers isomorphic
to $\widetilde{\fu}$.  Hence, $\widetilde{\mcM}$ is covered by open sets of the form $X_p \times S$, where $S$ is a smooth variety.
Since $X_p$ has rational singularities by Corollary \ref{XRat}, so does $\widetilde{\mcM}$.
\end{proof}

\begin{lemma} \label{tMNorm}

$\widetilde{\mcM}$ is the normalization of $T^*(G/P)$ in $K(\mcM)$.  

\end{lemma}

\begin{proof}
We must check that
$\widetilde{\mcM}$ is normal,  $K(\widetilde{\mcM}) = K(\mcM)$, and that
the morphism $\eta: \widetilde{\mcM} := G \times^P \tilde{\fu} \to  T^*(G/P)$ is integral and extends
the natural morphism $\Spec K(\widetilde{\mcM}) \to T^*(G/P)$.

First, $\widetilde{\mcM}$ is normal because it has rational singularities.
To show $K(\widetilde{\mcM}) = K(\mcM)$, we first show the equality of stabilizer groups
$G^{[1, \tilde{e}] }= G^e_0$.  Since $T^*(G/P) \to \ol{\mcO}_p$ is birational, the stabilizer $G^e$ is the same as the stabilizer $G^{[1 , e]}$ of $[1 , e] \in T^*(G/P)$, which is $P^e$.  Next, we claim that the stabilizer $P^e$ decomposes as $L^e U^e$.  Clearly $P^e \supset L^e U^e$.  For the reverse
inclusion, suppose $p = \ell u \in P^e$ with $\ell \in L$, $u \in U$.  Since $u e = e+e'$ where $e' \in [\fu, \fu]$, we have
$$
e = \ell u e = \ell(e + e') = \ell e + \ell e'.
$$
As $L$-modules, $\fu = V_p \oplus [\fu, \fu]$, so $\ell e \in V_p$ and $\ell e' \in [\fu, \fu]$.  We conclude that $e' = 0$ so $u \in U^e$, and
$e = \ell e$ so $\ell \in L^e$.  Hence $p \in L^e U^e$, proving the reverse inclusion.  We conclude that $P^e \supset L^e U^e$ as claimed.
Since $U^e$ is connected (being a unipotent group), we see that $G^e_0 = P^e_0 = L^e_0 U^e$.

On the other hand, $G^{[1 , \tilde{e}]} = P^{\tilde{e}} = P^{(e_p, e)} = P^{e_p} \cap P^e$.  Since $U$ acts trivially on $X_p$, we have
$P^{e_p} = L^{e_p} U = L^e_0 U$, where the last equality is by Lemma \ref{teStab}.  Hence $G^{[1, \tilde{e}] } = P^{e_p} \cap P^e = L^e_0 U^e = G^e_0$.

This equality of stabilizer groups implies that the $G$-orbit of $[1, \tilde{e}]$ is isomorphic to $\widetilde{\co} = G/G^e_0$.  Hence we have
$\widetilde{\co} \subset \widetilde{\mcM}$.  Since $\widetilde{\co}$ and $\widetilde{\mcM}$ have the same dimension, and $\widetilde{\mcM}$ is irreducible 
(since $\widetilde{\fu}$ is), $\widetilde{\co}$ is dense in $\widetilde{\mcM}$.  Since any orbit is open in its closure, $\widetilde{\co}$ is an open subset
of $\widetilde{\mcM}$.  Hence $K(\widetilde{\co}) = K(\widetilde{\mcM})$.  Since $\widetilde{\co}$ is also open in $\mcM$, $K(\widetilde{\co}) = K(\mcM)$.
Hence $K(\widetilde{\mcM}) = K(\mcM)$.

We now prove the assertions about $\eta$.  By construction, the morphism $X_p \to V_p$ is integral.  We can cover $G/P$ by open sets such that over each such open
set, the restriction of $\eta$ is the morphism $X_p \times S \to V_p \times S$, where the map is the identity on the second factor.
Hence $\eta$ is integral.
The fact that $\eta$ extends
the natural morphism $\Spec K(\widetilde{\mcM}) \to T^*(G/P)$ follows from the commutative diagram
$$
\begin{CD}
\widetilde{\mcM} @>{\eta}>> T^*(G/P) \\
@AAA @AAA \\
\widetilde{\co} @>>> \co \\
@AAA @AAA \\
\Spec K(\widetilde{\mcM}) @>>> \Spec K(\mcM).
\end{CD}
$$
\end{proof}

\begin{corollary}

There is a commutative diagram
% https://q.uiver.app/#q=WzAsNixbMiwwLCJcXHdpZGV0aWxkZXtcXG1jTX0gXFxjb25nIE4oXFxtY00gXFx0aW1lc197XFxvbHtcXG1jT31fcH0gVF4qKEcvUCkpIl0sWzQsMCwiVF4qKEcvUCkiXSxbNCwyLCJcXG9se1xcbWNPfV9wIl0sWzIsMiwiXFxtY00iXSxbMCwyLCJcXHRpbGRle1xcbWNPfV9wIFxcY29uZyBHL0deZV8wIl0sWzAsMCwiXFx0aWxkZXtcXG1jT31fcCBcXGNvbmcgRy9HXmVfMCJdLFsxLDIsIlxcbXUiLDJdLFswLDEsIlxcdGlsZGV7XFxldGF9Il0sWzMsMiwiXFxldGEiXSxbMCwzLCJcXHRpbGRle1xcbXV9IiwyXSxbNSwwLCIiLDIseyJzdHlsZSI6eyJ0YWlsIjp7Im5hbWUiOiJob29rIiwic2lkZSI6InRvcCJ9fX1dLFs0LDMsIiIsMCx7InN0eWxlIjp7InRhaWwiOnsibmFtZSI6Imhvb2siLCJzaWRlIjoidG9wIn19fV0sWzUsNCwiXFxjb25nIl1d
\[\begin{tikzcd}
	{\widetilde{\mcO}_p \cong G/G^e_0} && {\widetilde{\mcM} \cong N(\mcM \times_{\ol{\mcO}_p} T^*(G/P))} && {T^*(G/P)} \\
	\\
	{\widetilde{\mcO}_p \cong G/G^e_0} && \mcM && {\ol{\mcO}_p}
	\arrow[hook, from=1-1, to=1-3]
	\arrow["\cong", from=1-1, to=3-1]
	\arrow["{\tilde{\eta}}", from=1-3, to=1-5]
	\arrow["{\tilde{\mu}}"', from=1-3, to=3-3]
	\arrow["\mu"', from=1-5, to=3-5]
	\arrow[hook, from=3-1, to=3-3]
	\arrow["\eta", from=3-3, to=3-5]
\end{tikzcd}\]
where $\tilde{\mu}$ is proper and surjective. 

\end{corollary}

\begin{proof}

By Lemma \ref{jIso}, Lemma \ref{tMNorm} implies that $\widetilde{\mcM} \cong N(\mcM \times_{\ol{\mcO}_p} T^*(G/P))$, with $\tilde{\eta}$ and $\tilde{\mu}$ above the natural maps. Thus, $\tilde{\eta}$ is proper and surjective. Additionally, since $\eta$ and $\mu$ are $G$-equivariant, $\tilde{\eta}$ and $\tilde{\mu}$ are as well. Finally, because the fiber product and normalization are local constructions, and $\mu$ is an equivalence over $\mcO_p$, $\tilde{\mu}$ must be an equivalence over $\widetilde{\mcO}_p \cong G/G^e_0$.
\end{proof}

\section{The $G$-module structure of $R(\mcM)$} \label{s:Gmod}
In this section we will show that as a $G$-module, $R(\mcM) = \sum_j \Ind_L^G(\mu_j)$, where $L$ is as in the previous section and the $\mu_j$ are
weights defined below (see Theorem \ref{t:rings}).  
We also give an alternative expression replacing the $\mu_j$ by 
the dominant weights $\lambda_{j n/k}$ to which they are $W$-conjugate: we define a standard Levi subgroup $M$ conjugate to 
$L$ such that $R(\mcM) = \sum_j \Ind_M^G(\lambda_{j n/k})$ (see Theorem \ref{t:rings-dom}).

We retain the notation of the previous section, so $G = SL_n$ with center $Z \cong \Z_n$, and $k$ is the greatest common divisor of the parts of $p$.
For the remainder of the paper, $P = LU$ will denote the parabolic subgroup of $G$ defined in the previous section, and $Q = M U_Q$ will denote
 a particular parabolic subgroup defined below.  We will write a general parabolic subgroup with a notation such as $P_1$ or $P'$.

Let $\bar{\lambda}_1, \ldots, \bar{\lambda}_{k-1}$ denote the fundamental dominant weights for $SL_k$ (cf.~\cite{Hum}), and let $\bar{\lambda}_0 = 0$.   For each $j$, let 
$\bar{\mu}_j$ denote the unique weight of $T$ in the coset of $\bar\lambda_j$ modulo the root lattice of \(SL_k\) such that when $\bar{\mu}_j$ is expressed as a sum of simple
roots for $SL_k$, each coefficient $c$ satisfies $0 \leq c < 1$.  Thus, $\bar{\mu}_0 = \bar{\lambda}_0 = 0$.
The weights $\bar{\lambda}_j$ and $\bar{\mu}_j$ are conjugate under
the Weyl group of $SL_k$ \cite[Proposition 5.9]{Gra2022}. Let $\lambda_1, \ldots, \lambda_{n-1}$ denote the fundamental weights for $G = SL_n$, and set $\lambda_0 = 0$.
 Let $e^{\mu_j}$ denote the character of $L$ defined by $e^{\mu_j} = \rho^* e^{ \bar{\mu}_j}$ (see \eqref{rhoDef} for the definition of $\rho$). The corresponding weight $\mu_j$ is $W$-conjugate
to $\lambda_{j n/k}$ for $j \neq 0$, and $\mu_0 = \lambda_0 = 0$.

\begin{example} \label{Ex:k=6}
Suppose $k = 6$ and $t = (t_1, \ldots, t_6) \in T$, with $\prod t_j = 1$.  Then
$e^{\bar{\lambda}_j}(t) = t_1 \cdots t_j$.
We have $\bar{\lambda}_j = \bar{\mu}_j$ for $j = 1, 5$,
and 
$$
e^{\bar{\mu}_2}(t) = t_1 t_4, \hspace{.2in} e^{\bar{\mu}_3}(t) = t_1 t_3 t_5, \hspace{.2in} e^{\bar{\mu}_4}(t) = t_1 t_2 t_4 t_5.
$$
Further examples can be found in  \cite[Example 3.4]{Gra2022} and \cite[Example 3.8]{GPR}.
\end{example}

If $\mu$ is a weight defining a character of $L$ which
is trivial on $L^e_0$, then
$e^{\mu}$ defines a character $\chi$ of $A(\co) \cong L^e/L^e_0 \cong G^e/G^e_0$.
In this case, we say $e^{\mu}$ (or simply $\mu$) is a lift of $\chi$
to $L$ (or $P$).  In particular, this holds for the weight $\mu_j$;
we denote by $\chi_j$ the corresponding character of $A(\co)$.

\begin{lemma} \label{l:component}
The set 
$\{ \chi_j \}$ is the set of 
characters of $A(\co)$.  
\end{lemma}

\begin{proof}
The restrictions of the $e^{\bar{\mu}_j}$ (or $e^{\bar{\lambda}_j}$) to the center $Z_k$ of $SL_k$ give the characters
of $Z_k$.  In light of Remark \ref{r:center}, this implies that the restrictions of the $e^{\mu_j}$ to $L^e$
yield the set of characters of the component group $G^e/G^e_0 \cong \Z_k$.
\end{proof}

\begin{lemma} \label{l:P-mod}
As a $P$-module,
$$
R(\widetilde{\fu}) \cong \bigoplus_{j=0}^{k-1} R(\fu) \otimes \C_{\mu_j}.
$$
\end{lemma}

\begin{proof}
The ring $R(\V)$ is a free $R(\V_{ad})$ module with basis given by the $e^{-\bar{\mu}_j}$ (see \cite[Proposition 3.3]{Gra2022}, noting that the negative sign comes from our difference in convention).  The function $e^{-\bar{\mu}_j}$ is a $T$-weight vector with weight $\bar{\mu}_j$, so
as a $T$-module, $R(\V) \cong \bigoplus_j R(\V_{ad}) \otimes \C_{\bar{\mu}_j}$.  Therefore, as an $L$-module,
\begin{align*}
R(X_p) & = R(V_p) \otimes_{R(\V_{ad})} R(\V) \cong R(V_p) \otimes_{R(\V_{ad})} (\oplus_j R(\V_{ad}) \otimes \C_{\rho^*\bar{\mu}_j}) \\
& \cong R(V_p) \otimes (\oplus_j \C_{\mu_j}) \cong \bigoplus_j R(V_p) \otimes \C_{\mu_j}.
\end{align*}
Since the $L$-action on $X$ is inflated to a $P$-action where $U$ acts trivially, this is an isomorphism of $P$-modules.
Since $\tilde{\fu} = X_p \times_{V_p} \fu$,  we have a $P$-module isomorphism
%\begin{align*}
$$
R(\tilde{\fu}) =  R(X_p) \otimes_{R(V_p)} R(\fu) \cong ( \bigoplus_j R(V_p) \otimes \C_{\mu_j} ) \otimes_{R(V_p)} R(\fu) =  \bigoplus_j R(\fu) \otimes \C_{\mu_j},
$$
as desired.
\end{proof}

There is a left action of $G$ on $\widetilde{\co}$ arising from left multiplication, and this induces a left action of $G$ on $R(\mcM) = R(\widetilde{\co})$ given by $(g \cdot \varphi) (x) = \varphi(g^{-1} x)$
for $g \in G$, $\varphi \in R(\mcM)$, $x \in \widetilde{\co}$.
There is also a right action of $A(\co)$ on $\widetilde{\co}$ arising from right multiplication, and this induces a right action of $A(\co)$ on $R(\widetilde{\co})$ given
by $(\varphi \cdot a) (x) = \varphi(x a^{-1})$ for $a \in A(\co)$.  Under the natural map $f: Z \to A(\co)$ (which is surjective
because we are in type $A$), the left and right actions are compatible, in the sense that
$z \cdot \varphi = \varphi \cdot f(z)$.  

Write $R(\mcM) = \oplus_j R(\mcM)_{\chi_j}$
for the decomposition of $R(\mcM)$ into $A(\co)$-isotypic components.
By the above remarks, we can view this as a decomposition into $Z$-isotypic components under
the left $Z$-action on $R(\mcM)$, if we view the characters $\chi_j$ as characters of $Z$ via pullback
along $f: Z \to A(\co)$.  

\begin{theorem} \label{t:rings}
As $G$-modules, 
\begin{equation} \label{e:rings}
R(\mcM) = \sum_{j=0}^{k-1} \Ind_L^G(\mu_j),
\end{equation}
where the $j$-th summand equals the isotypic component $R(\mcM)_{\chi_j}$.
\end{theorem}

\begin{proof}
We adapt the strategy of \cite{McG}.
Since $\tilde{\mu}: \widetilde{\mcM} \to \mcM$ is proper, $R(\mcM) = R(\widetilde{\mcM}) = H^0(\widetilde{\mcM}, \co_{\widetilde{\mcM}})$.
The variety $\widetilde{\mcM}$ has rational singularities by
Lemma \ref{l:Mrational}.  Also, $\mcM$ has rational singularities by \cite[Cor.~6.3]{Bro:98}, or alternatively because $\mcM$ has symplectic singularities (see \cite[Lemma 2.5]{Lo21}),
so it has rational singularities by \cite[Proposition 1.3]{Beau}. Since $\widetilde{\mcM}$ and $\mcM$ both have rational singularities, for any resolution $\widehat{\mu}: \widehat{\mcM} \to \widetilde{\mcM}$, we have equalities $R\widehat{\mu}_* \co_{\widehat{\mcM}} = \co_{\widetilde{\mcM}}$ and $R(\widetilde{\mu} \circ \widehat{\mu})_* \co_{\widehat{\mcM}} = \co_{\mcM}$ in the derived category of coherent sheaves. Since $R(\widetilde{\mu} \circ \widehat{\mu})_* = R\widetilde{\mu}_* \circ R\widehat{\mu}_*$, this implies $R\widetilde{\mu}_* \co_{\widetilde{\mcM}} = \co_{\mcM}$, which further implies that $R^i \widetilde{\mu}_* \co_{\widetilde{\mcM}} = 0$ for $i > 0$.
Since $\mcM$ is affine, $H^0(\mcM,R^i \widetilde{\mu}_* \co_{\widetilde{\mcM}}) \cong H^i(\widetilde{\mcM}, \co_{\widetilde{\mcM}})$, so $H^i(\widetilde{\mcM}, \co_{\widetilde{\mcM}}) = 0$ for $i > 0$.

Let $\pi: \widetilde{\mcM} \to G/P$ denote the projection.    Since $\pi$ is an affine morphism, $R^i \pi_* \co_{\widetilde{\mcM}} = 0$ for
$i>0$, and so
$$
H^i (\widetilde{\mcM}, \co_{\widetilde{\mcM}}) = H^i (G/P, \pi_* \co_{\widetilde{\mcM}}) = H^i(G/P, R(\widetilde{\fu}) ).
$$
Lemma \ref{l:P-mod} then implies that
\begin{equation} \label{e:coh-equal}
H^i (\widetilde{\mcM}, \co_{\widetilde{\mcM}}) \cong  \bigoplus_{j=0}^{k-1} H^i(G/P, R(\fu) \otimes \C_{\mu_j}).
\end{equation}
All terms in this equation vanish for $i>0$.  Therefore, as $G$-modules,
\begin{align*}
R(\mcM) & = H^0 (\widetilde{\mcM}, \co_{\widetilde{\mcM}})  = \sum (-1)^i H^i (\widetilde{\mcM}, \co_{\widetilde{\mcM}})  \\ 
& = \sum_j \sum_i (-1)^i H^i(G/P, R(\fu) \otimes \C_{\mu_j})  = \sum_j \Ind_L^G (\mu_j),
\end{align*}
where the last equality is by \cite[Lemma 2.1]{McG}.  This proves \eqref{e:rings}.  Since each character $e^{\mu_j}$ of $L$ restricts to the character $\chi_j$ of $Z$, the representation $\Ind_L^G (\mu_j)$ of $G$ must lie entirely in the $\chi_j$ isotypic component as a representation of $Z$. The theorem follows.
\end{proof}

The formula of Theorem \ref{t:rings} can be rewritten
using different weights or different Levi subgroups.
Suppose $L' = g L g^{-1}$ for $g \in G$.  
If $e^{\mu}$ is a character of $L$, then $e^{\mu'} : = e^{\mu} \circ C_{g^{-1}}$
is a character of $L'$, where $C_{g^{-1}}$ denotes conjugation by
$g^{-1}$.  We say $(L' ,\mu')$ is conjugate by $g$ to $(L, \mu)$.  In this case, 
there is a $G$-module
isomorphism 
\begin{equation}  \label{e:ind-conj}
\Ind_L^G(\mu) \cong \Ind_{L'}^G(\mu').
\end{equation}

We will apply this in the following situation.  
By definition, $L$ is the standard Levi subgroup of block diagonal matrices with blocks of sizes $p_1', p_2', \ldots $ (in decreasing order).  Any subgroup consisting of block diagonal
matrices with blocks of sizes $\{ p_1', p_2', \ldots \}$ in any fixed order is conjugate to $L$.  Let $M$ be the subgroup where the blocks occur in the following order.
Every entry of $p'$ occurs with multiplicity divisible by $k$.  Let $d$ be the partition of $n/k$ with the same parts as $p'$, but in which each part has $\frac{1}{k}$ its multltiplicity
in $p'$.  Let $q$ be the sequence obtained by concatenating $k$ copies of $d$, and let $M$ be the block diagonal matrix with block sizes $q_1, q_2, \ldots, q_{p_1}$, in that order.
Let $Q = M U_Q$ denote the standard parabolic subgroup containing $M$ as Levi factor.

\begin{theorem} \label{t:rings-dom}
With notation as above, 
as $G$-modules
$$
R(\mcM) = \sum_{j = 0}^{k-1} \Ind_M^G(\lambda_{jn/k})
$$
where the $j$-th summand equals the isotypic component $R(\mcM)_{\chi_j}$.
\end{theorem}

\begin{proof}
For each $j$,
there is an element $g_j \in G$ such that the pair $(M, \lambda_{jn/k})$ is conjugate by $g_j$ to $(L, \mu_j)$.
\end{proof}

\begin{example} \label{Ex64.1}
Let $p = (6,4)$, so $p' = (2^4, 1^2)$.  Then $A = (A_1, A_2, A_3, A_4, a_5, a_6)$, where the first $4$ matrices are square of size $2$,
We have $k = 2$, $d = (2^2, 1)$, and $q =(2,2,1,2,2,1)$.  We can find $g \in G$ such that
$(M, e^{\lambda_{5}})$ is conjugate by $g$ to $(L, e^{\mu_1})$.  For example, $g$ can be chosen
to conjugate $A = (A_1, A_2, A_3, A_4, a_5, a_6)$ to $(A_1, A_3, a_5, A_2, A_4, a_6)$.  
%It follows that as $G$-modules, $R(\mcM) = \Ind_M^G(0) + \Ind_M^G(\lambda_2)$.
The element $g$ is not unique; for example we could also choose $g$ to conjugate $A$ to $(A_3, A_1, a_5, A_2, A_4, a_6)$.
\end{example}

\begin{example} \label{Ex12-6}
Let $p = (12, 6)$, so $p' = (2^6, 1^6)$.  
Suppose $A =  (A_1, A_2, \dots, A_6, a_7, \dots, a_{12})$ is an element of $L$, where the $A_i$
are $2 \times 2$ blocks and the $a_i$ are scalars.  By definition, $e^{\mu_i} = \rho^* e^{\bar{\mu}_i}$.
In this example, $k = 6$.  Using the definition of $\rho$ and the expressions for the $e^{\bar{\mu}_i}$ 
given in Example \ref{Ex:k=6},
it is straightforward to compute $e^{\mu_i}(A)$; for example, $e^{\mu_1}(A) = \det A_1 \cdot a_7$
and $e^{\mu_2}(A) = \det A_1 \det A_4 \cdot a_7 a_{10}$.  
Thus, $\mu_1 = (1^2, 0^{10}, 1, 0^5)$ is $W$-conjugate to $\lambda_3 = (1^3, 0^{15})$,
and $\mu_2 = (1^2, 0^4, 1^2, 0^4, 1, 0^2, 1, 0^2)$ is $W$-conjugate to $\lambda_6 = (1^6, 0^{12})$.
We can express the $\mu_i$ as a linear combination of the $\lambda_j$, with coefficients
given by the inner products of $\mu_i$ with the roots $\alpha_j$,
so for example  $\mu_1 = \lambda_2 - \lambda_{11} + \lambda_{12}$.  In this example,
the Levi subgroup $M$ has block sizes given by the sequence $(2,1)$, concatenated with itself $6$ times.
For comparison, we note there is a standard Jacobson-Morozov parabolic whose Levi subgroup
has block sizes $(1^3, 2^6, 1^3)$.
\end{example}

\begin{remark} \label{r:T-ind}
The formula for $R(\mcM)$ can be expressed as a linear combination of terms of the form $\Ind_H^G (\nu)$, as was done
for $R(\co)$ in ~\cite[Corollary 3.2]{McG}.  Indeed, it follows from the proof of \cite[Proposition 2.6]{PA} that for any highest weight representation $V^L_{\lambda}$ of $L$ (in particular for
any character $e^{\lambda}$) of $L$, we have the following equation of virtual $G$-modules:
\begin{equation} \label{e:T-ind}
\Ind_L^G(V^L_{\lambda}) = \sum_{w \in W_L} (-1)^{\ell(w)} \Ind_{H}^G(\lambda + \rho_L - w\rho_L).
\end{equation}
This formulation allows us to use Frobenius reciprocity 
to compute the multiplicity of $V_{\mu}$ in $R(\mcM)$ in terms of dimensions of weight spaces of $H$.
This is convenient because these dimensions are readily computable, for example, by the familiar Kostant multiplicity
formula.  Note that the multiplicity of $V_{\mu}$ in $R(\mcM)$
can be computed more directly by
using a generalization of Kostant's multiplicity formula
which applies to subgroups such as $L$ (see \cite[Proposition 3.4]{Vog78}).
\end{remark}

\section{Lifting and conjugation} \label{s:lifting}
Theorems \ref{t:rings} and \ref{t:rings-dom} provide two complete descriptions of $R(\mcM)$ as a $G \times A(\co)$-module. 
However, $R(\mcM)$ has additional structure: it is a ring and an $R(\co)$-module, and we will see in the next section that it has a natural grading.  
In order to study this additional structure, we will need to show that the lifting of characters of $A(\co)$ is compatible
with conjugation of Levi subgroups.
In the next section we will apply this to the the pairs $(L, \mu_j)$ and $(M, \lambda_{j n/k})$, which are conjugate
by some $g \in G$.
In particular, we know that $e^{\mu_j}$ lifts the character
$\chi_j$ of $A(\co)$ to $P$; we will use the results of this section to deduce that $e^{\lambda_{j n/k}}$ lifts $\chi_j$ to $Q$.
The difficulty is that $P$ and $Q$ are not
conjugate subgroups of $G$; only their Levi subgroups are conjugate.  
Thus, while we can easily show that $e^{\lambda_{j n/k}}$ defines a lift of $\chi_j$
to $C_g P$ (cf.~Proposition \ref{p:conjugate}), we need to prove that we can lift to $Q$.
This follows from the key result of this section, Proposition
\ref{p:independent};  the discussion before this
proposition explains the issues that need to be dealt with in the proof.

\begin{lemma} \label{l:stab}
Let $P_1 = L_1 U_1$ be a Richardson parabolic for $\co$ and let $e_1 \in \co \cap \fu_1$.  Then $G^{e_1} = P_1^{e_1}$.
\end{lemma}

\begin{proof}
Since $T^*(G/P_1) \to \ol{\mcO}_p$ is birational, the stabilizer $G^{e_1}$ is the same as the stabilizer $G^{[1 , e_1]}$ of $[1 , e_1] \in T^*(G/P_1)$, which is $P_1^{e_1}$
(cf.~ the proof of Lemma \ref{tMNorm}).
\end{proof}

\begin{definition} \label{d:P-lift}
Let $P_1$ be a Richardson parabolic for $\co$.  We say a character $e^{\nu}$ of 
$P_1$ lifts the character $\chi$ of $A(\co)$  if for some $e_1 \in \co \cap \fu_1$, the character $e^{\mu}$ is trivial on $(P_1^{e_1})_0$, and
the restriction of the restriction of $e^{\nu}$ to $P_1^{e_1}$ defines the character $\chi$
of $P_1^{e_1}/(P_1^{e_1})_0 \cong A(\co)$.  
\end{definition}

The notion of lifting to $P_1$ is 
independent of the choice of $e_1$.  Indeed, $\co \cap \fu_1$ is a single
$P_1$-orbit (see \cite{Ric}), so any other element of $\co \cap \fu_1$ is of the form $g e_1$ for some
$g \in P_1$, and $P_1^{g e_1}= C_g P_1^{e_1}$.  Thus, $C_g$ induces
an isomorphism of $P_1^{e_1}/(P_1^{e_1})_0$ with  $P_1^{ge_1}/(P_1^{ge_1})_0$.
Since $e^{\nu}$ is a class function on $P_1$, $e^{\nu} \circ C_g = e^{\nu}$.  Thus,
$e^{\nu}$ induces the same function on $P_1^{e_1}/(P_1^{e_1})_0$ and 
$P_1^{ge_1}/(P_1^{ge_1})_0$ when we identify them via $C_g$.

\begin{definition} \label{d:Levi-lift}
We say that a character $e^{\nu}$ of a Levi subgroup
$L_1$ lifts the character $\chi$ of $A(\co)$ if there is a Richardson parabolic $P_1$ for $\co$
such that $L_1$ is a Levi factor of $P_1$, and such that the extension of $e^{\nu}$ to $P_1$
lifts $\chi$.  In this setting, we also refer to $\nu$ as a lift of $\chi$ to $L_1$.
\end{definition}

The next result shows that lifting, in the sense of the previous definition,
is compatible with conjugation of Levi subgroups.

\begin{proposition} \label{p:conjugate}
Suppose $e^{\nu_1}$ defines a character of $L_1$ lifting a character $\chi$
of $A(\co)$.  Suppose $(L_1, \nu_1)$ is conjugate via $g$
to $(L_2, \nu_2)$.  Then $e^{\nu_2}$ is a character of $L_2$ lifting $\chi$.
\end{proposition}

\begin{proof}
Choose a Richardson parabolic $P_1$ containing $L_1$ as Levi factor, and $e_1 \in \co \cap \fu_1$. Then $e^{\nu_1}$ extends to a character of $P_1$ that induces the character $\chi$ on $P_1^{e_1}/(P_1^{e_1})_0 \cong A(\co)$. Additionally, $P_2 := C_g P_1$ is a parabolic containing $L_2$ and $e_2 := g e_1 \in \co \cap \fu_2$. Therefore, $e^{\nu_2} = e^{\nu_1} \circ C_{g^{-1}}$ defines a character of $P_2$ lifting the character $\chi$ on $P_2^{e_2}/(P_2^{e_2})_0 \cong A(\co)$. By definition, this means $e^{\nu_2}$ is a character of $L_2$ lifting $\chi$.
\end{proof}

The next proposition is the key result of this section.
It shows that lifting in the sense of Definition \ref{d:Levi-lift}
is independent of the choice
of Richardson parabolic containing $L_1$.   The issue is the following.
Suppose $P_1$ and $P_2$ are Richardson parabolics containing $L_1$ as a Levi subgroup,
and $e_i \in \co \cap \fu_i$ for $i = 1,2$.
A necessary
condition for $e^{\nu}$ to lift $\chi$ to $P_j$
is that $e^{\nu}$ is trivial on $(P^{e_j}_j)_0$.
However, $e_1$ and $e_2$ need not be conjugate by $L_1$, so the same is true for
the stabilizer groups $P^{e_j}_j$.
Thus, even if $e^{\nu}$ is trivial on $(P^{e_1}_1)_0$, it is not obvious that
it is trivial on $(P^{e_2}_2)_0$.  In the proof of
Proposition \ref{p:independent}, we circumvent this issue by showing
that we can choose $e_1$ and $e_2$ so that $(P^{e_1}_1)_0$
and $(P^{e_2}_2)_0$ share a common reductive part.

\begin{proposition} \label{p:independent}
Suppose that $P_1$ and $P_2$ are Richardson parabolics for $\co$, and both $P_1$ and $P_2$
contain $L_1$ as a Levi subgroup.  If the extension of $e^{\nu}$ to $P_1$ lifts
$\chi$, then the extension of $e^{\nu}$ to $P_2$ lifts $\chi$.
\end{proposition}

\begin{proof}
Without loss of generality, we can assume $L_1 = L$ since we can pick some $g \in G$ such that $C_g L_1 = L$ and conjugate all of the data by $g$.  We can compare
both $P_1$ and $P_2$ to $P$, so it suffices to prove the proposition under the assumption $P_1 = P$.  
We will define a particular reductive part $L^e_{red}$ of $P^e_0$ (where $e$ is as in the previous section) below, and then show that there exists $e_2 \in \co \cap \fu_2$ such
that $L^e_{red}$ is also a reductive part of $(P_2^{e_2})_0$.  

The Levi $L$ canonically decomposes the standard representation $\C^n$ of $G$ into simple subrepresentations $V_1, \dots, V_{p_1}$ of $L$.  Fix standard $H$-eigenvectors $X_{ij} \in \C^n$ for each $i \in \{1, \dots, p_1\}$ and $j \in \{1, \dots, \dim V_i\}$, such that $X_{i1}, \dots, X_{i,\dim V_i}$ form a basis for $V_i$, and are ordered in accordance with the positive system. We note that the choice of $e$ in Section \ref{s:Springer} maps each $X_{ij}$ to $X_{i + 1,j}$ if $i < p_i$ and $X_{ij}$ to zero if $i = p_i$.  

It will be convenient to rephrase the definition of $e$ as follows.   The block decompositions of the matrices $A_i$ and $A_1$ in \eqref{e:blocks1} and \eqref{e:blocks2}
are reflected in decompositions of each of the subspaces $V_i$.  Let $u_1 > u_2 > \cdots > u_{\ell}$ be the distinct parts of $p$, and let $m_r$ denote the multiplicity of $u_r$ in $p$.
Let $r_i$ denote the largest index such that $u_{r_i} \geq i$.  Then $V_i = V_{i,1} \oplus \cdots \oplus V_{i, r_i}$,
where $\dim V_{i,r} = m_r$, and the $V_{i,r}$ are spanned by standard basis vectors such that the decomposition is compatible with the positive system.
Note that $V_i \cong V_{i+1}$ if $p$ has no parts of size $i$.  The element $e$ then takes $V_{i, j}$ isomorphically onto $V_{i+1, j}$ if $j \leq r_{i+1}$, and
is given by the identity matrix with respect to our chosen basis of each space.  If $j > r_{i+1}$ then $e$ takes $V_{i, j}$ to $0$.  
The decomposition is illustrated in Example \ref{ex:decomp} below.  Observe that for fixed $r$, the sum of the spaces $V_{i,r}$ is the subspace of $\C^n$
corresponding to the Jordan blocks of $e$ of size $u_r$, and there are $m_r$ such blocks.  This confirms that $e$ has Jordan form given by $p$.

The group $L$ consists of the elements of $G$ which preserve each subspace $V_i$.  We specify a reductive part $L^e_{red}$ of $L^e_0$ to be the subgroup
of $L^e_0$ defined by setting all the $B_j$ matrices in \eqref{e:blocks2} to $0$.  (Note that the $C_j$ matrices satisfy \eqref{Le0} since we are considering the identity
component.)  This implies that an element
$g \in L^e_0$ preserves each of the subspaces $V_{i,j}$, and as a map from $V_{i,j}$ to itself is given by the matrix $C_j$.  Note that $L^e_{red}$ 
is also a reductive part of $P^e_0 = G^e_0$.

Each choice of a parabolic containing $L$ corresponds to an ordering of the summands $V_1, \dots, V_{p_1}$.   In particular, given a permutation $\sigma \in S_{p_1}$, we consider the partial flag $F_{\sigma} = \{0 \subset V_{\sigma(1)} \subset V_{\sigma(1)} \oplus V_{\sigma(2)} \subset \dots \subset \C^n\}$, and we let $P_{\sigma}$ be the parabolic subgroup $P_{\sigma} := \Stab_{G}(F_{\sigma})$. 
We have $P_2 = P_{\sigma}$ for some $\sigma \in S_{p_1}$.
Additionally, we specify an element $e_{\sigma} \in \fu_{P_{\sigma}} \cap \co$ by fixing its action on each $X_{ij}$: when $i < p_i$, $e_{\sigma}$ maps each $X_{\sigma(i)j}$ to $X_{\sigma(i + s)j}$, where $s$ is the smallest positive integer such that there is a basis vector $X_{\sigma(i + s)j}$, and when $i = p_i$, $e_{\sigma}$ maps $X_{\sigma(i)j}$ to zero. 
In terms of the decompositions of the $V_i$, we see that $e_{\sigma}$ takes the subspace $V_{\sigma(i), j}$ isomorphically to $V_{\sigma(i+s), j}$, where $s$ is the smallest
positive integer such that $j \leq r_{\sigma(i+s)}$, or to $0$ if no such $j$ exists.  Reasoning as above shows that $e_{\sigma}$ has Jordan form given by $p$.

From the explicit descriptions of $e_{\sigma}$ and $L^e_{red}$, we see that $L^e_{red}$ is contained in $L^{e_{\sigma}}_0 \subset G^{e_{\sigma}}_0$.  
Since the groups $G^e_0$ and $G^{e_{\sigma}}_0$ are conjugate subgroups of $G$, they have isomorphic reductive parts.
Hence $L^e_{red}$ is
a reductive part of $G^{e_{\sigma}}_0$, and hence of $(P_{\sigma}^{e_{\sigma}})_0$.  

Since the extension of $e^{\nu}$ to $P$ defines a character $\chi$ of $A(\co)$,  the restriction of $e^{\nu}$ to $P^e_0$ is trivial, so the restriction
of $e^{\nu}$ to $L^e_{red}$ is trivial.  Since $L^e_{red}$ is a reductive part of $(P_{\sigma}^{e_{\sigma}})_0$, this implies that the extension of
$e^{\nu}$ to $P_{\sigma}$ is trivial on $(P_{\sigma}^{e_{\sigma}})_0$.  Hence the extension of $e^{\nu}$ to $P_{\sigma}$ defines a character of
$A(\co)$.  This character must be $\chi$ since it is determined by the restriction of $e^{\nu}$ to the center of $G$.
\end{proof}

\begin{example} \label{ex:decomp}
Let $p = (10^2, 6, 4^2)$, so $p' = (5^4, 3^2, 2^4)$, $u = (10, 6, 4)$, and $(m_1, m_2, m_3) = (2,1,2)$.  Then
for $i = 1,2,3,4$ we have $V_i = V_{i, 1} \oplus V_{i,2} \oplus V_{i,3}$ where $V_{i,r} \cong \C^{m_r}$;
for $i = 5,6$, we have $V_i = V_{i, 1} \oplus V_{i,2}$, and for $i = 7,8,9,10$ we have $V_i = V_{i, 1}$.
\end{example}

The following lemma, and graded versions of the lemma
(see \eqref{e:gradingP} and \eqref{e:gradingQ}), 
will be essential in studying the structure of $R(\mcM)$ as a graded ring and $R(\co)$-module.  
For the principal orbit, this result is in \cite{Bry}; an explanation can also be found
in \cite{Gra1992}.  A similar result is implicitly used in \cite{Som}.  For the parabolic $P = LU$ and the weight $\mu_j$, this 
inclusion can be seen from the proof of Theorem \ref{t:rings} (and in fact it is an isomorphism in this case).  
We provide a proof for the convenience of the reader.

\begin{lemma} \label{lem:inclusion}

Let $P'$ be a Richardson parabolic
for $\co$ with Levi subgroup $L'$.
Let $e^{\mu'}$ be a character of $L'$ (equivalently of $P'$) lifting the character $\chi$ of $A(\co)$. Then there is
an inclusion of $R(\co)$-submodules of $R(\mcM)$:
\begin{equation} \label{e:inclusion}
H^0(G/P', R(\fu_{P'}) \otimes \C_{\mu'}) \subset R(\mcM)_{\chi}.
\end{equation}

\end{lemma}

\begin{proof}
Since $L'$ is conjugate to $L$, we may assume $L' = L$.  To simplify notation assume $P' = P$ and $\mu' = \mu$.
Let $L_{\mu}$ be the line bundle on $G/P$ corresponding to the representation $\C_{\mu}$ of $P$. Consider the composition 
$$
\co \cong G/G^e \xrightarrow{i} T^*(G/P) \xrightarrow{\pi} G/P,
$$
where $i$ is the natural inclusion and $\pi$ is the natural projection. It can be deduced from the projection formula for the map $\pi$ that $H^0(T^*(G/P), \pi^* L_{\mu}) \cong H^0(G/P, R(\fu) \otimes \C_{\mu})$. Furthermore, since $\co \subset T^*(G/P)$ is an open subset and $T^*(G/P)$ is integral, restriction provides an inclusion 
\begin{equation} \label{resO}
H^0(T^*(G/P), \pi^* L_{\mu}) \subset H^0(\co, \pi^* L_{\mu}|_{\co}).
\end{equation}
However, the composition $\pi \circ i$ is the natural projection $p: G/G^e \to G/P$, so the line bundle $\pi^* L_{\mu}|_{\co}$ is simply the bundle associated to the representation $\C_{\mu}$ restricted to $G^e$. 
The hypothesis that $e^{\mu}$ lifts $\chi$ to $P$ implies that the restriction of $\C_{\mu}$ to $G^e$ is the pullback to $G^e$ of the representation $\chi$ of $G^e/G^e_0 \cong A(\co)$. 
Therefore, $H^0(\co, p^* L_{\co}) \cong \Ind_{G^e}^G(\chi)$, which is the $\chi$-isotypic component $R(\mcM)_{\chi}$ of $R(\mcM) \cong \Ind_{G^e_0}^G(0)$. Thus \eqref{resO} is the desired inclusion, and since restriction induces an isomorphism on $R(T^*(G/P)) = R(\co)$, it is an inclusion of $R(\co)$-modules.
\end{proof}

\begin{remark} \label{r:vanish}
It follows from Lemma \ref{lem:inclusion} and the proof of Theorem \ref{t:rings} that the equations
\begin{equation} \label{e:description}
 R(\mcM)_{\chi} = H^0(G/P, R(\fu) \otimes \C_{\mu}) = \Ind_L^G (\mu)
 \end{equation} 
 will hold precisely when
two conditions are satisfied: The inclusion \eqref{e:inclusion} is an isomorphism (for $P' = P$),
and $H^i(G/P, R(\fu) \otimes \C_{\mu}) = 0$ for $i>0$.  These two conditions
tend to pull in opposite directions.  
If $\mu$ is dominant, then
the cohomology vanishing holds by \cite[Theorem 2.2]{Bro}.  It also holds
for certain nondominant $\mu$.  For example, suppose
$\co$ is principal.  Then $P = B$, and Hesselink's result from \cite{He76}
can be applied to show that
\begin{equation} \label{e:Hesselink}
H^i(G/B, R(\fu) \otimes \C_{w \lambda_j}) = 0 \mbox{  for  } i>0, w \in W
\end{equation}
(see \cite[Theorem 1.3]{Gra1992}).
On the other hand, write $\mu = \sum a_i \ga_i$ where the
$\ga_i$ are the simple roots.  If all $a_i \leq 0$, then the inclusion
\eqref{e:inclusion} is an isomorphism (see \cite{Gra1992}), but this also holds for some other $\mu$.
In fact, in the case of the principal orbit, 
the $w \lambda_j$ are exactly the $\mu$ for which both cohomology vanishing holds,
and \eqref{e:inclusion} is
an isomorphism.  
The cohomology vanishing for the
nondominant weights $\mu_j$ is one of the main consequences of the
geometric constructions of this paper.  
\end{remark}

We conclude this section by discussing the pairs $(L', \nu)$ which are conjugate to $(L, \mu_j)$ or equivalently, to $(M, \lambda_{jn/k})$.
Note that given a fixed $L'$, the possible weights
$\nu$ form a single orbit under the 
relative Weyl group $W(L') := N_G(L')/L'$.  In the following discussion, we assume that the partition $p$ is fixed, and all Levi subgroups considered
are Levi factors of Richardson parabolics for $\co = \co_p$.

\begin{definition}

Let $(A_1, \dots, A_{p_1})$ be the block decomposition of $A \in L$. For $j \in \{0, \dots, k - 1\}$, we say a subset $S \subset \{1, \dots, p_1\}$ is $j$-equally distributed if for every distinct part $a$ of $p'$, with multiplicity $n_{a}$, there are $\frac{j n_a}{k}$ elements of $S$ such that $p'_i = a$. We say that $\nu$ is a $j$-equally distributed weight of $L$ (or simply that the pair $(L, \nu)$ is $j$-equally distributed) if $\nu$ defines a character $e^{\nu}$ of $L$ and $e^{\nu}(A)$ is a product of determinants $\prod_{i \in S} \det A_i$ for some $j$-equally distributed $S \subset \{1, \dots, p_1\}$.  Additionally, if $L'$ is a Levi conjugate to $L$, and $\nu$ is a weight of $L'$, we say the pair $(L', \nu')$ is $j$-equally distributed if it is conjugate to a $j$-equally distributed pair of the form $(L, \nu)$. Finally, we shall say that a pair $(L', \nu)$ is equally distributed if it is $j$-equally distributed for some $j \in \{0, \dots, k - 1\}$.

\end{definition}

Suppose $L'$ is a standard Levi subgroup.  Then $L'$ is block diagonal, and a weight $\nu$ is equally distributed if and only if $e^{\nu}$ is a product of determinants of a proportional number of blocks of each size.
The definition above forces the weight $\nu$ in any equally distributed pair $(L', \nu)$ to be minuscule (or $0$ if $j=0$).   
We will say that a lift $e^{\nu}$ of $\chi$ to $L'$ is minimal if the length of $\nu$ (with respect to a $W$-invariant inner product
on the real span of the weights of $G$) is minimal.  If $L'$ is not assumed standard, we say the lift is minimal $(L', \mu)$ is conjugate to a pair where the Levi subgroup is standard and the lift is minimal.

\begin{proposition} \label{jlifts}

Given a Levi subgroup $L'$ and a weight $\nu$ that defines a character $e^{\nu}$ of $L'$, the following three conditions are equivalent:

\begin{enumerate}

\item The pair $(L', \nu)$ is conjugate to $(L, \mu_j)$.
\item The weight $\nu$ is a $j$-equally distributed weight of $L'$.
\item The character $e^{\nu}$ is a minimal lift of the character $\chi_j$ of $A(\co)$ to $L'$.
\end{enumerate}

\end{proposition}

\begin{proof}
If $j = 0$, then $nu = \mu_j = 0$ and $\chi_0$ is trivial, and  the equivalence is straightforward, so we assume $j \neq 0$.
We may assume that $L' = L$ since all three conditions are preserved by conjugation.  Furthermore, we assume $\nu$ is minuscule since all three conditions require this.
For property (3), this is true because if $e^{\nu}$ is a lift of $\chi_j$ to $L$, then it must induce the central character given by pulling $\chi_j$ back along the surjection $Z \to A(\co)$, which we also refer to as $\chi_j$. 
We have shown that the minuscule weight $\mu_j$ is a lift of $\chi_j$ to $L$.  Since the minuscule weights are minimal among all weights with a given central character, if $\nu$ is a minimal lift of $\chi_j$,
then it must be minuscule.  Since the minuscule weights which induce a given central character form a single $W$-orbit, if (3) holds, then $\nu$ must be in the $W$-orbit of $\mu_j$.

Properties (1) and (2) are equivalent, since the number of determinants of each size appearing in a character $e^{\nu}$ is invariant under conjugation. We now show that (2) and (3) are equivalent.
We can check if $(L, \nu)$ satisfies (3) using the choice of $P$ and $e$ in Section \ref{s:Springer}. We will use the explicit description of $L^e_0$ in \eqref{Le0} to show that $e^{\nu}$ is trivial on $L^e_0$ if and only if $\nu$ is equally distributed, and that $e^{\nu}$ lifts $\chi_j$ if and only if $\nu$ is $j$-equally distributed. 

Under the assumptions of either (2) or (3), 
$\nu$ is minuscule and defines a character $e^{\nu}$ of $L$, so we can expand $e^{\nu}(A)$ as $\prod_{i \in S} \det A_i$ for some proper subset $S \subset \{1, \dots, p_1\}$, where $A_1, \dots A_{p_1}$ are the blocks of $A \in L$. We can further expand $e^{\nu}(A)$ in terms of determinants of the matrices $C_i$ described in Section \ref{s:Springer}, obtaining an equation of the form 
\begin{equation} \label{e-nu}
e^{\nu}(A) = \det(C_1)^{a_1} \det(C_2)^{a_2}  \dots \det(C_{\ell})^{a_{\ell}}.
\end{equation}

Recall that $u_1 > u_2 \dots > u_{\ell}$ are the distinct parts of $p$. 
Equation \eqref{Le0} implies that $e^{\nu}$ is trivial on $L^e_0$ if and only if there is a fixed $j \in \{0, \dots, k - 1\}$ such that $a_i = \frac{j u_i}{k}$ for each $i \in \{1, \dots \ell \}$. In this case, we can check by restricting to the center of $G$ that $e^{\nu}$ lifts the character $\chi_j$ of $A(\co)$. 

Let $t_1, \dots, t_{\ell}$ be the distinct parts of $p'$ in decreasing order.  The multiplicity of the largest part $t_1$ is $n_1 = u_{\ell}$, the next largest part $t_2$ is $n_2 = u_{\ell -1} - u_{\ell}$, and
in general, the multiplicity $n_i$ of $t_i$ is $u_{\ell - i + 1} - u_{\ell - i + 2}$ (with $u_{\ell + 1} = 0$). 
Recall also from the explicit description of $L^e$ in Section \ref{s:Springer} that each $C_i$ appears in the first $u_i$ blocks of $L$. 

Assume $(L, \nu)$ satisfies (2). Then for each $i \in \{1, \dots \ell \}$, $S$ must contain $\frac{j u_i}{k}$ out of the entries $\{1, \dots, u_i\}$. In particular, we will have $a_i = \frac{j u_i}{k}$ in \eqref{e-nu} and thus $e^{\nu}$ is trivial on $L^e_0$, and in particular lifts $\chi_j$. It is a minimal lift because $\nu$ is minuscule, so (2) implies (3). 

Finally, assume $(L, \nu)$ satisfies (3). Thus, in \eqref{e-nu} we have $a_i = \frac{j u_i}{k}$ for each $i \in \{1, \dots \ell \}$. We shall show by induction that for every $t_i$, there are $\frac{j n_{i}}{k}$ elements $s \in S$ such that $p'_s = t_i$. For the base case $i = 1$, recall that $C_{\ell}$ only appears in the first $n_1 = u_{\ell}$ blocks of $L$, so the equation $a_{\ell} = \frac{j u_{\ell}}{k}$ implies that $S$ must contain $\frac{j n_1}{k}$ of the first $n_1$ entries. These are precisely the entries $s \in S$ with $p'_s = t_1$. For the inductive step, assume we have proved the inductive hypothesis for $t_1, \dots, t_{i - 1}$. The matrix $C_{\ell - i + 1}$ only appears in the first $u_{\ell - i + 1}$ blocks of $L$, so the equation $a_{\ell - i + 1} = \frac{j u_{\ell - i + 1}}{k}$ implies that $S$ must contain $\frac{j u_{\ell - i + 1}}{k}$ of the first $u_{\ell - i + 1}$ entries. These are precisely the entries such that $A_i$ has size at least $t_i$. By the inductive hypothesis, $S$ already contains $\frac{j u_{\ell - i + 2}}{k}$ of the first $u_{\ell - i + 2}$ entries, since there are the ones corresponding to blocks of size larger than $t_i$. Furthermore, there are $n_i = u_{\ell - i + 2} - u_{\ell - i + 1}$ blocks of size $t_i$, so $S$ must contain exactly $\frac{j n_i}{k}$ indices $s$ with $p'_s = t_i$. Thus, (3) implies (2).
\end{proof}

Note that Proposition \ref{jlifts} relates the characters appearing in our paper to the description in \cite[Conjecture 4.2]{Som} when the nilpotent orbit $\mcO$ is even. Furthermore, it follows from \eqref{e:ind-conj} and Proposition \ref{jlifts} that $R(\mcM)_{\chi_j} \cong \Ind_{L'}^G(\nu)$ whenever $(L', \nu)$ is $j$-equally distributed. We will explore graded versions of these isomorphisms in the next two sections. 

\begin{example} \label{Ex-lift}
Suppose $G = SL_{9}$ and $p = (6, 3)$.  Then $p' = (2^3, 1^3)$.  The Levi subgroup $L$ has block sizes $(2^3, 1^3)$ while the Levi subgroup $M$ has block sizes $(2,1,2,1,2,1)$.  As characters of the torus, $\mu_1$ and $\mu_2$ are given by $(1^2, 0^4, 1, 0^2)$ and $(1^4, 0^2, 1^2, 0)$, respectively. We can see explicitly that the pairs $(L, \mu_1)$ and $(L, \mu_2)$ are conjugate to $(M, \lambda_3)$ and $(M, \lambda_6)$ respectively. As a further example, if $(L', \nu)$ is 1-equally distributed pair with $L'$ standard, then $L'$ has three blocks of size two and three blocks of size one, and the character $e^{\nu}$ is a product of one block of each size.

We note that given a pair $(L', \nu)$, the equation $R(\mcM)_{\chi_j} \cong \Ind_{L'}^G(\nu)$ can fail if the pair is not conjugate to $(L, \mu_j)$, even if the individual constituents $L'$ and $\nu$ are. For example, consider the weight $\nu = (0^6, 1^3)$ and the pair $(L, \nu)$, which is not equally distributed even though $\nu$ is conjugate to $\mu_1$. One can use the tools described in Remark \ref{r:T-ind} to show that in this case $\Ind_L^G(\mu_1) \neq \Ind_L^G(\nu)$. In particular, the representation $V_{\lambda}$ with highest weight $\lambda := (2,1^2,0^6)$ has multiplicity one in $\Ind_L^G(\mu_1)$ and multiplicity two in $\Ind_L^G(\nu)$. Note that in this example $e^{\nu}$ does not lift a character of $A(\co)$, and this computation shows that the inclusion in Lemma \ref{e:inclusion} cannot hold. We can construct a similar example with $\nu$ dominant by taking $L'$ to be standard with block sizes $(1^3, 2^3)$, and letting $\nu = \lambda_3$. 
\end{example}

\section{The grading on $R(\mcM)$} \label{s:grading}
In this section we study the structure of $R(\mcM)$ as a graded $G$-module,
making use of Theorem \ref{t:rings}.  We continue to write $P = LU$ and
$Q = M U_Q$ for the specific parabolic subgroups defined previously.  Lemma \ref{lem:inclusion} shows
that there is an inclusion of the spaces of sections of line bundles on certain cotangent bundles
as submodules of $A(\co)$-isotypic components
of $R(\mcM)$.  The proof of Theorem \ref{t:rings} shows that for
the line bundles on $T^*(G/P)$ arising from the weights $\mu_j$,
this inclusion is an isomorphism, and the higher cohomology
vanishes.  In this section we sharpen this
statement by incorporating the graded structure, and also extend
it to $T^*(G/Q)$ with the weights $\lambda_{j n/k}$ (see Theorem \ref{t:graded}).
As is known, such results, combined
with Theorem 3.9 and Remark 3.11 of Panyushev's paper \cite{Pan}, are sufficient to describe
the graded $G$-module structure of $R(\mcM)$ in terms of generalizations
of Lusztig's $q$-analogue of Kostant's multiplicity formula.  In addition, we
identify the degrees in which minuscule representations occur in $R(\mcM)$,
and combining our work with a theorem of Grantcharov \cite{Grant}, we describe
the minimal embedding of $\mcM$.  Finally, we show that the natural map
$\mcM \to \overline{\co}$ is a bijection over the boundary of $\co$.

The grading on $R(\mcM)$ arises by lifting the square of the
dilation $\C^*$-action on $\co$ to an action
on $\widetilde{\co}$  (see \cite{BrylinskiKostant1994}).  We will use the following
equivalent formulation.
Choose $h \in \fh$ such that $[h,e] = - 2e$
(the negative sign is included because our nilpotent element lies in the negative Borel).  Define a vector field $\Xi$ on $\widetilde{\co} = G/G^e_0$ by the equation
$$
(\Xi \phi) (g G^e_0) = - \tfrac{1}{2} \tfrac{d}{dt} \phi(g \exp (th) G^e_0) |_{t=0}.
$$
This extends the definition of \cite{Gra1992}, except that we have included an extra factor of $\frac{1}{2}$
to make it compatible with the grading on $R(\overline{\co})$.  
This factor means that our grading is $\frac{1}{2}$ times the grading used in many
references, e.g.~\cite{Bry}, \cite{BrylinskiKostant1994}, \cite{Gra1992}, \cite{Som}.

We say that $\phi \in R^d(\mcM)$ if $\Xi \phi = d \phi$; here 
$d \in \frac{1}{2} \Z$.  The right action of the component group $A(\co)$ on $R(\mcM)$ is compatible with the grading on $R(\mcM)$.  
In type $A$ the compatibility is particularly easy to see, since the natural map $Z \to A(\co)$ (where $Z$ is
the center of $G$) is surjective, so there are representatives in $Z$ of elements of $A(\co)$. 

The grading on $R(\mcM)$ is compatible with the grading on the subring $R(\overline{\co})$ 
inherited from the polynomial ring $R(\fg)$ with generators in degree $1$.  This was proved by Brylinski for the principal orbit (\cite[Lemma 2.4]{Bry}); a proof
in language closer to this paper can be found in \cite[Prop.~2.2]{Gra1992}.   The proof carries through almost unchanged
to arbitrary orbits.  As observed in \cite[Remark 1.5]{BrylinskiKostant1994}, the grading on $R(\mcM)$ is the unique $G$-invariant
grading compatible with the grading on $R(\co)$, and moreover, the grading
is nonnegative.  Note that Brylinski and Kostant work with arbitrary $G$-invariant orbit covers, not only the universal
cover, and this characterization of the grading is valid in that setting.  We will make use of this
in the proof of Theorem \ref{t:bijection}.  

Brylinski shows that when $e^{\mu}$ is a character of $L$ lifting the character $\chi$ of $A(\co)$, Lemma \ref{lem:inclusion} can be upgraded to a graded inclusion
 \begin{equation} \label{e:gradingP}
H^0(G/P, R^d(\fu) \otimes \C_{\mu}) \subset R^{d + \mu(h)/2}(\mcM)_{\chi}.
\end{equation}
See \cite[Prop.~3.3]{Bry}; a proof can also be found in \cite[Prop.~2.1]{Gra1992}.
Although these references deal with the principal orbit, the proof extends to arbitrary orbits.
The analogous statement holds if we replace $P = LU$ by the parabolic
$Q = M U_Q$.  Choose $e' \in \co \cap \fu_Q$ and $h' \in \fh$ such that
$[h', e'] = - 2 e'$.  Let $e^{\mu'}$ be a character of $M$ lifting the character $\chi$ of $A(\co)$.  Then 
we can similarly upgrade Lemma \ref{lem:inclusion} to a graded inclusion
\begin{equation} \label{e:gradingQ}
H^0(G/Q, R^d(\fu_Q) \otimes \C_{\mu'}) \subset R^{d + \mu'(h')/2}(\mcM)_{\chi}.
\end{equation}
%Using these facts, we have the following graded version of \eqref{e:cohzero-equal}. 
These facts imply the following graded analogues of \eqref{e:coh-equal}.

\begin{theorem} \label{t:graded}
Set $d_j = d - \frac{1}{2} \mu_j(h)$ and $d'_j = d - \frac{1}{2} \lambda_{jn/k}(h')$.  Then
\begin{align}
R^d(\mcM)  & \cong  \bigoplus_{j=0}^{k-1} H^0(G/P, R^{d_j}(\fu) \otimes \C_{\mu_j})   \label{e:graded1}  \\ 
 & \cong  \bigoplus_{j=0}^{k-1} H^0(G/Q, R^{d'_j}(\fu_Q) \otimes \C_{\lambda_{jn/k}}). \label{e:graded2}
\end{align}
The isotypic component $R^d(\mcM)_{\chi_j}$ corresponds to the $j$-th term in either sum on the right.
\end{theorem}

\begin{proof}
The only surviving terms in \eqref{e:coh-equal} occur when $i = 0$; then the left hand side is $R(\mcM)$, so
\begin{equation} \label{e:cohzero-equal}
R(\mcM) \cong  \bigoplus_{j=0}^{k-1} H^0(G/P, R(\fu) \otimes \C_{\mu_j}).
\end{equation}
The isomorphism \eqref{e:graded1} follows from this and \eqref{e:gradingP}.  For \eqref{e:graded2},
we have a sequence of $G$-module isomorphisms
\begin{equation} \label{L-to-M}
H^0(G/P, R(\fu) \otimes \C_{\mu_j}) \cong \Ind_L^G(\mu_j) \cong \Ind_{M}^G (\lambda_{j n/k}) \cong H^0(G/Q, R(\fu_Q) \otimes \C_{\lambda_{j n/k}}).
\end{equation}
Here, the first isomorphism holds because $H^i(G/P, R(\fu) \otimes \C_{\mu_j}) = 0$ for $i>0$, and the third isomorphism holds because $\lambda_{j n/k}$ is dominant, so
$H^i(G/Q, R(\fu_Q) \otimes \C_{\lambda_{j n/k}}) = 0$ for $i>0$ by \cite[Theorem 2.2]{Bro}.  The middle isomorphism
follows from \eqref{e:ind-conj} since $(M, \lambda_{j n/k})$ and $(L, \mu_j)$ are conjugate. 
Combining \eqref{e:cohzero-equal} with the relations between gradings described in \eqref{e:gradingP} and  \eqref{e:gradingQ} shows
that $H^0(G/P, R^{d_j}(\fu) \otimes \C_{\mu_j}) \cong  H^0(G/Q, R^{d'_j}(\fu_Q) \otimes \C_{\lambda_{jn/k}})$.  Combining this
with \eqref{e:graded1} yields \eqref{e:graded2}.  The statement about isotypic components follows from \eqref{e:gradingP} and  \eqref{e:gradingQ}.
\end{proof}

\begin{remark} \label{r:dominant}
The proof of the preceding theorem generalizes to give an analogous result 
for the $\chi_j$-isotypic part $R(\mcM)_{\chi_j}$ (for any $j$), in which 
$Q = M U_Q$ is replaced by any standard parabolic 
$P' = L' U'$ such that
$(L', \lambda_{j n/k})$ is conjugate to $(L, \mu_j)$.
The reason is that $\lambda_{j n/k}$ is dominant, so the
necessary vanishing of higher cohomology holds by \cite[Theorem 2.2]{Bro}.
\end{remark}

\begin{remark} \label{r:vanish2}
A key point in the proof of \eqref{e:graded2} 
 is that the inclusion $H^0(G/Q, R(\fu_Q) \otimes \C_{\lambda_{j n/k}}) \subset R(\mcM)_{\chi_j}$
is an isomorphism (cf~Remark \ref{r:vanish}).  In fact, we
need our main result Theorem \ref{t:rings} to
show that this inclusion is an isomorphism.  
\end{remark}

If $\co$ is understood, we write $\deg (V_{\lambda_{j n/k}})$ for the degree in which 
the minuscule representation $V_{\lambda_{j n/k}}$ occurs in $R(\mcM)$.  
The next corollary and the following proposition identify this degree, generalizing \cite{Gra1992} for the principal orbit.

\begin{corollary} \label{c:minuscule}
The representation $V_{\lambda_{j n/k}}$ occurs exactly once
in $R(\mcM)$, in degree $\tfrac{1}{2} \lambda_{j n/k} (h')$,
where $h'$ is as above.  Moreover, if $V$ is another irreducible
$G$-module occurring in $R^d(\mcM)_{\chi_j}$, then 
$d > \tfrac{1}{2} \lambda_{j n/k} (h')$.
\end{corollary}

\begin{proof}
The representation $V_{\lambda_{j n/k}}$ can occur in \eqref{e:graded2} only in the summand
$\Ind_{M}^G (\lambda_{j n/k})$, because this is the only summand 
which transforms under $Z$ by the character
$\chi_j = e^{\lambda_{j n/k} }|_Z$ by which $Z$ acts on $V_{\lambda_{j n/k}}$.
Frobenius reciprocity implies that the multiplicity
of $V_{\lambda_{j n/k}}$ in $\Ind_{M}^G (\lambda_{j n/k})$ equals
the dimension of the $\lambda_{j n/k}$-isotypic component of $V_{\lambda_{j n/k}} |_M$.
This isotypic component is the highest weight space, which is $1$-dimensional,
so $V_{\lambda_{j n/k}}$ occurs once in $R(\mcM)$.  This representation
occurs in $R(\mcM)$ as
$$
H^0(G/Q, \C_{\lambda_{j n/k}}) = H^0(G/Q, R^0(\fu_Q) \otimes \C_{\lambda_{j n/k}}),
$$
so by Theorem \ref{t:graded}, it occurs in degree $\tfrac{1}{2} \lambda_{j n/k} (h')$.
For the second statement, if $V$ occurs in $R^d(\mcM)_{\chi_j}$,
then it must occur in $H^0(G/Q, R^i(\fu_Q) \otimes \C_{\lambda_{j n/k}})$ for
some $i>0$.  By \eqref{e:gradingQ}, $d = i + \tfrac{1}{2} \lambda_{j n/k} (h') > \tfrac{1}{2} \lambda_{j n/k} (h') $.
\end{proof}

\begin{proposition} \label{p:closed-formula}
\begin{equation}\label{e:closed formula}
\lambda_{jn/k}(h') =  j(k-j) \sum_{i=1}^t (p_i/k)^2 = \frac{j}{k} (1 - \frac{j}{k}) \sum_{i=1}^t p_i^2,
\end{equation}
where $t$ is the number of parts of $p$.  Hence, for $r \in \Q$,
\begin{equation}\label{e:degree-relation2}
\deg (V_{\lambda_{r n}}) = \frac{1}{2} r (1-r) \sum_{i=1}^t p_i^2.
\end{equation}
Therefore,
\begin{equation}\label{e:degree-relation}
\deg (V_{\lambda_{j n/k}}) = \deg (V_{\bar{\lambda}_j})  \sum_{i=1}^t(p_i/k)^2
\end{equation}
where on the right hand side, $\deg (V_{\bar{\lambda}_j}) $ refers to the degree
with respect to the principal nilpotent orbit in $SL_k$. 
\end{proposition}

\begin{proof}
If $d$ is the partition used in the definition of $M$, then its transpose $d'$ is $\tfrac{1}{k} p = (\tfrac{p_1}{k}, \tfrac{p_2}{k}, \ldots, )$.  Divide the
Young diagram of $p$ into the shape $q'$, which we view as $k$ copies of the Young diagram of $d'$, by sliding boxes to the right.

Now fill the boxes of $q'$ with the entries $1, \ldots, n$ by filling in numbers consecutively down columns, and then proceeding along
columns from left to right.  An element \(e'\in\mcO\cap\fu_Q\) is given by the sum of $E_{sr}$, where \(s\) immediately follows \(r\) in a row of \(q'\).
An example is given below for concreteness.  

\begin{example}
Let \(p=(6,3^2)\) be a partition of \(n=12\) with \(k=3\).
Then \(p\) on the left and \(q'=(d',d',d') \) on the right is illustrated by
\[
\begin{tikzpicture}[
    x=0.65cm,
    y=0.65cm,
    every node/.style={font=\small}
]

% -------------------------------------------------
% The Young diagram p = (6^2,4^2,2^2)
% -------------------------------------------------
\begin{scope}
    \foreach \row/\length in {
        0/6,
        1/3,
        2/3
    }{
        \foreach \col in {0,...,\numexpr\length-1\relax}{
            \draw (\col,-\row) rectangle ++(1,-1);
        }
    }
\end{scope}

% -------------------------------------------------
% First copy of d' = (2,1^2)
% -------------------------------------------------
\begin{scope}[xshift=8cm]
    \foreach \row/\length in {
        0/2,
        1/1,
        2/1
    }{
        \foreach \col in {0,...,\numexpr\length-1\relax}{
            \draw (\col,-\row) rectangle ++(1,-1);
        }
    }

    % First column
    \foreach \row/\entry in {
        0/1,
        1/2,
        2/3
    }{
        \node at (0.5,-\row-0.5) {\entry};
    }

    % Second column
    \foreach \row/\entry in {
        0/4
    }{
        \node at (1.5,-\row-0.5) {\entry};
    }
\end{scope}

% -------------------------------------------------
% Second copy of d' = (2,1^2)
% -------------------------------------------------
\begin{scope}[xshift=9.7cm]
    \foreach \row/\length in {
        0/2,
        1/1,
        2/1
    }{
        \foreach \col in {0,...,\numexpr\length-1\relax}{
            \draw (\col,-\row) rectangle ++(1,-1);
        }
    }

    % First column
    \foreach \row/\entry in {
        0/5,
        1/6,
        2/7
    }{
        \node at (0.5,-\row-0.5) {\entry};
    }

    % Second column
    \foreach \row/\entry in {
        0/8
    }{
        \node at (1.5,-\row-0.5) {\entry};
    }
\end{scope}

% -------------------------------------------------
% Third copy of d' = (2,1^2)
% -------------------------------------------------
\begin{scope}[xshift=11.4cm]
    \foreach \row/\length in {
        0/2,
        1/1,
        2/1
    }{
        \foreach \col in {0,...,\numexpr\length-1\relax}{
            \draw (\col,-\row) rectangle ++(1,-1);
        }
    }

    % First column
    \foreach \row/\entry in {
        0/9,
        1/10,
        2/11
    }{
        \node at (0.5,-\row-0.5) {\entry};
    }

    % Second column
    \foreach \row/\entry in {
        0/12
    }{
        \node at (1.5,-\row-0.5) {\entry};
    }
\end{scope}
\end{tikzpicture}.
\]
Note that boxes in different copies of $d'$ can still be viewed as on the same row in $q'$.
For example, $2$ and $6$ are in the same row of $q'$, as are $8$ and $9$,
and $E_{62}$ and $E_{98}$ are terms in the expression for $e'$.
In this example above, \(n/k=4\), and \(\lambda_{4}(h')=12=\lambda_{8}(h')\).
\end{example}

Now define $h' \in \fh$ as follows.  Fill in the boxes of $q'$ in a different way, by filling in the
boxes in row $i$ with $p_i -1, p_i -3, \ldots, - (p_1 - 1)$.  Then $h'_i$ is the number
in the second filling in the box with $i$ in the first filling.
In our example, \(h'=(5,2,2,3,1,0,0,-1,-3,-2,-2,-5)\).  By construction, $[h', e'] = - 2 e'$.

We now prove \eqref{e:closed formula}.  Observe that $\lambda_{j n/k} (h')$
is the sum of the entries in the second filling in the first $j$ copies of $d'$.
We evaluate this sum as follows.

First, consider the entries in the $i$-th row of the first copy of $d'$.  There are $\tfrac{p_i}{k}$ entries,
and summing them gives
$$
(p_i - 1) + (p_i - 3) + \cdots + (p_i - (2 \frac{p_i}{k} - 1) ) = \frac{p_i^2}{k^2} (k-1).
$$
If we do the same for the second copy of $d'$, we get 
$$
(p_i - 1 - 2 \frac{p_i}{k} ) + (p_i - 3- 2 \frac{p_i}{k} ) + \cdots + (p_i - (2 \frac{p_i}{k} - 1)- 2 \frac{p_i}{k}  ) = \frac{p_i^2}{k^2} (k-3).
$$
Similarly, the entries in the $i$-th row of the $r$-th copy of $d'$ add to $\frac{p_i^2}{k^2} (k-(2r-1))$.  When we
sum these over the first $j$ copies of $d'$, we get
$$
\frac{p_i^2}{k^2} \sum_{r = 1}^j (k - (2r-1)) = \frac{p_i^2}{k^2} j (k-j).
$$
Taking the sum of this over the $t$ rows of $q$ gives \eqref{e:closed formula}; then
\eqref{e:degree-relation2} is an immediate consequence.  Equation
\eqref{e:degree-relation} follows because $\deg (V_{\bar{\lambda}_j})$
can be calculated by taking $n=k$ and $p = (k)$ in \eqref{e:closed formula},
yielding $\deg (V_{\bar{\lambda}_j}) =\frac{1}{2}  j (k-j)$.
%(as also follows from \cite[Prop.~2.3]{Gra1992}).
\end{proof}

Using Theorem \ref{t:graded} and a theorem of Grantcharov (see \cite{Grant}), we can describe generators 
of $R(\mcM)_{\chi_j}$ as an $R(\co)$-module.

\begin{theorem} \label{t:generate}
$R(\mcM)_{\chi_j}$ is generated as $R(\co)$-module
by the minuscule representation $V_{\lambda_{jn/k}}$ it contains. 
\end{theorem}

\begin{proof}
It follows from \cite{Grant} that the natural map
$$
R(\fg) \otimes V_{\lambda_{j n/k}} \to H^0(T^*(G/Q), R(\fu_Q) \otimes \C_{\lambda_{j n/k}})
$$
is surjective.  By Theorem \ref{t:graded}, the target is $R(\mcM)_{\chi_j}$.  The map factors through the natural map
$R(\fg) \to R(\co)$.  Hence the map
$$
R(\co) \otimes V_{\lambda_{j n/k}} \to R(\mcM)_{\chi_j}
$$
is surjective.
\end{proof}

Let $\fn$ denote the maximal ideal of $R(\mcM)$ generated by homogeneous elements of positive degree.
If $\fn = \fa \oplus \fn^2$ is a $G$-module decomposition, then there is an embedding
$\mcM \hookrightarrow \fa^*$, which is the minimal embedding of $\mcM$ into a $G$-representation
(see \cite{BrylinskiKostant1994}, Lemma 4.10 and Proposition 4.11).
The following corollary gives a precise description of the minimal embedding.  Note that
$R(\co)$ contains $\fg^*$ in degree $1$ (since $R(\co)$ is a quotient
of $R(\fg) = S(\fg^*)$). Since $\fg^* \cong \fg$ as $G$-modules, we
view $\fg$ as a subspace of $R^1(\co)$.

\begin{corollary} \label{c:minimal}
If $n >2$, then 
$$
\fa  =  \fg \oplus (\bigoplus_j V_{\lambda_{j n/k}}).
$$
\end{corollary}

\begin{proof}
The inclusion $\fa \subset  \fg\oplus (\oplus_j V_{\lambda_{j n/k}})$ follows from
the fact that $V_{\lambda_{j n/k}}$ generates $R(\mcM)_{\chi_j}$ over $R(\co)$.
By examining the degrees in which the representations $V_{\lambda_{j n/k}}$
occur, one can show that $V_{\lambda_{j n/k}}$ is not in $\fn^2$, so the
inclusion is an equality.  We omit the details, which are similar to 
the proof for the principal orbit given in \cite[Theorem 4.1]{Gra1992}.
\end{proof}

\begin{remark} \label{r:low-degree}
It is of interest to understand \(R^d(\widetilde{\mcO})\) for \(d=\tfrac{1}{2}\) and \(d=1\).
These spaces can be studied for any cover of $\mcO$, but the functions on any cover are a subspace
of the functions on the universal cover, so we restrict our attention to the universal cover.
From \eqref{e:closed formula}, we see that \(\deg(V_{\lambda_{jn/k}})= \tfrac{1}{2}\) if and only if \(\widetilde\mcO\) is the universal cover of the principal orbit in \(SL_2\) with \(j=1\).
The space
\(R^1(\mcM)\) is a simple Lie algebra $\fg'$ containing $\fg$; when $\fg'$ is strictly larger than $\fg$, it gives rise
to one of the ``shared orbit pairs" classified by Brylinski and Kostant (see \cite{BrylinskiKostant1994}).
Corollary \ref{c:minuscule} implies that $\fg'$ is larger than $\fg$ exactly when
\(\deg(V_{\lambda_{jn/k}})=1\) for some $j$.  From \eqref{e:closed formula}, we see that this occurs 
if and only if \(\widetilde\mcO\) is the universal cover of the principal orbit in \(SL_3\) with \(j=1,2\), or \(\widetilde\mcO\) is the universal cover for \(p=(2^2)\) with \(j=1\).  Thus,
these are the only examples for $G = SL_n$ where $\fg'$ is larger than $\fg$.  This
is consistent with the classification of shared orbit pairs in \cite{BrylinskiKostant1994}.  
\end{remark}

\begin{lemma} \label{l:degree-boundary}
Let $v$ be a partition of $n$ different from $p$, such that the orbit $\co_v$ is contained in $\overline{\co}$, where $\co = \co_p$.  Write $R(\mcM_p)$ and $R(\mcM_v)$ for the rings of functions on the universal covers of $\co_p$ and $\co_v$.  Suppose $V$ is a minuscule representation occuring in $R^{d_p}(\mcM_p)$.  If $V$ occurs in $R^{d_v}(\mcM_v)$ for some $d_v$, then
$d_p > d_v$.
\end{lemma}

\begin{proof}
By \eqref{e:degree-relation2},  it suffices to show that if $\co_v \subset \overline{\co}_p$ and $v \neq p$, then
$\sum p_i^2 > \sum v_i^2$. 

Recall that a box move is given by replacing a partition $(r_1, r_2, \dots, r_{\ell})$ with $(r_1, r_2, \dots r_i - 1 \dots r_j + 1 \dots, r_{\ell})$ or $(r_1, r_2, \dots r_i - 1  \dots, r_{\ell}, 1)$. Equivalently, a box move is given by moving a single box in the associated Young diagram to a lower row. The hypothesis that $\co_v \subset \overline{\co_p} - \co_p$ is equivalent to $p > v$ in the dominance order. In particular, this means we can obtain $v$ by performing a finite number of box moves on $p$. Thus, it is sufficient to show that a box move decreases the sum of the squares of the parts of a partition. However, this is equivalent to the condition that if $a -2 \geq b \geq 0$, then $a^2 + b^2 > (a - 1)^2 + (b + 1)^2$. This is true since 

\vspace{4pt}

\end{proof}

\vspace{-46pt}

$$
a^2 + b^2 - (a - 1)^2 + (b + 1)^2 = 2(a - b - 1) > 0. 
$$

For the principal nilpotent orbit, the following result was observed by Brylinski and Kostant (see \cite[Remark 4.13]{BrylinskiKostant1994}).  

\begin{theorem} \label{t:bijection}
The natural map $\pi: \mcM \to \overline{\co}$ is bijective over the boundary of $\co$.
\end{theorem}

\begin{proof}
It suffices to show that if $\co_v$ is contained in the boundary of $\co = \co_p$, then $\pi^{-1}(\co_v)$ consists of a single $G$-orbit
which is isomorphic to $\co_v$.
Since $\pi$ is finite, $\pi^{-1}(\co_v)$ is a finite union of $G$-orbits, each of which is the same
dimension as $\co_v$.  In particular, each of these is a covering of $\co_v$.  

We claim that each of these orbits is isomorphic to $\co_v$.
Indeed, let $\co'_v$ be one such orbit.  
The ring $R(\co'_v)$ is a graded subring of $R(\mcM_v)$.
On the other hand, $R(\co'_v)$ is a quotient of $R(\mcM_p)$.
The orbit $\co'_v$ is invariant under the $\C^*$-action on $\mcM_p$
which lifts the square of the dilation action on $\fg$.  This action
induces a grading on $R(\co'_v)$, and the quotient map
$f: R(\mcM_p) \to R(\co'_v)$ is a map of graded rings.  The gradings
on $R(\co'_v)$ obtained in these two ways---as a subring of $R(\mcM_v)$,
and as a quotient of $R(\mcM_p)$---coincide, since both
gradings are $G$-invariant and compatible with the grading on $S(\fg)$,
and by \cite[Remark 1.5]{BrylinskiKostant1994}, these properties characterize
the grading on $R(\co'_v)$. 

Since $R(\co'_v)$ is a graded $G$-invariant subring of $R(\mcM_v)$, Lemma \ref{l:degree-boundary}
implies that each of the minuscule representations 
in $R(\mcM_p)$ occurs (if at all) in $R(\co'_v)$
in degree strictly less than in $R(\mcM_p)$.  Hence the minuscule
representations all map to $0$ under $f$.  Theorem \ref{t:generate} then implies that for each nontrivial character $\chi_j$ of $A(\co)$, the map
$f$ takes the isotypic component $R(\mcM_p)_{\chi_j}$ to $0$.  Thus, the image of $f$ is contained in the image under $f$ of the subring $R(\overline{\co}_p)$ of $R(\mcM_p)$.
But this image lies in the subring $R(\overline{\co}_v)$ of $R(\mcM_v)$.  We conclude that $R(\co'_v) = R(\overline{\co}_v)$.  Hence 
$\co'_v$ is isomorphic to $\co_v$, proving the claim.

The map $\pi$ is a quotient map by $A(\co_p)$, and the left action of $Z$ is
compatible with the right action of $A(\co_p)$.  Thus, $Z$ acts transitively on the set of irreducible components of $\pi^{-1}(\co_v)$.
On the other hand, $Z$ acts trivially on $\co_v$, so it is contained in the stabilizer of any point of $\co_v$.  Since
$\co_v \cong \co'_v$,  $Z$ is contained in the stabilizer of
any point of $\co'_v$.  Therefore $Z$ fixes the orbit $\co'_v$.  We conclude that $\pi^{-1}(\co_v)$ consists of a single $G$-orbit.  
By the first paragraph, that orbit is isomorphic to $\co_v$.  The theorem follows.
\end{proof}

\section{A vanishing conjecture} \label{s:vanishing}
In this section, we state a cohomology vanishing conjecture generalizing the
vanishing proved in this paper.   We show that this conjecture
would yield an alternative proof of Theorem \ref{t:generate}. Finally,
we discuss the relation of our results to \cite[Conjecture 4.2]{Som}.

\begin{conjecture} \label{conj:vanishing}
Let $P' = L'U'$ be a Richardson parabolic for $\co$
such that $(L', \mu')$ is conjugate to $(L, \mu_j)$.
Then
$$
H^i(G/P', R(\fu') \otimes \C_{\mu'}) = 0 \mbox{    for    } i>0.
$$
\end{conjecture}

The conjecture would have the following consequence.
With the notation of the conjecture, let $e' \in \co \cap \fu'$ and  $h' \in \fh$ such that $[h',e'] = - 2e'$.

\begin{proposition} \label{p:conjecture}
Suppose $(L', \mu')$ is conjugate to $(L, \mu_j)$
Then $R(\mcM)_{\chi_j} = \Ind_{L'}^G(\mu')$, and if Conjecture \ref{conj:vanishing} holds, then
\begin{equation} \label{e:graded-van-char}
R^d(\mcM)_{\chi_j} = H^0(G/P', R^{d - \mu'(h')/2} (\fu') \otimes \C_{\mu'}).
\end{equation}
\end{proposition}

\begin{proof}
The proof of Theorem \ref{t:graded} applies almost unchanged.
\end{proof}

\begin{remark}
Conjecture \ref{conj:vanishing} is largely motivated by the case where $\co$ is principal, since
then $P' = B$ and the conjecture holds by \eqref{e:Hesselink}.  It also holds
for $P'=P$ and $\mu' = \mu_j$, as shown in the proof of 
Theorem \ref{t:rings}.  However, cohomology vanishing also holds in certain cases where the hypothesis
that $(L', \mu')$ is conjugate to $(L, \mu_j)$ is not assumed.  For example, let
$\mu'$ be any $L$-dominant weight and let 
$F_{\mu'}$ denote the irreducible representation of $L'$
with highest weight $\mu'$.  
If $\mu'$ is also dominant for $G$, then
\begin{equation} \label{e:vanish-remark}
H^i(G/P', R(\fu') \otimes F_{\mu'}) = 0 \mbox{    for    } i>0,
\end{equation}
by \cite[Theorem 2.2]{Bro} if $\dim F_{\mu'} = 1$ (for $G$ of any type),
and by \cite[Theorem 2.3]{Grant} for any $F_{\mu'}$ (in type $A$).  In a different
direction, we can replace
the assumption that $(L', \mu')$ is conjugate to $(L, \mu_j)$
by the weaker assumption that   
$\mu'$ is $W$-conjugate to $\mu_j$.  Some additional
hypothesis is then required for \eqref{e:vanish-remark} to hold, since Grantcharov has provided an example in
which $\dim F_{\mu'} > 1$ and \eqref{e:vanish-remark} fails.
We can ask if either of the following suffices:
\begin{enumerate}
\item $\dim F_{\mu'} = 1$
\item $e^{\lambda_{j n/k}}$ defines an equally distributed character
of $L'$, and $F_{\mu'}$ is an irreducible constituent of $V_{\lambda_{j n/k}} |_{L'}$.
\end{enumerate}
The second condition is exactly what would be required to prove
Theorem \ref{t:generate} by adapting the argument for the principal orbit
given in \cite[Theorem 3.2]{Gra1992}.
This would be different from the proof of Theorem \ref{t:generate} in this paper, which relies on Grantcharov's result \cite[Theorem 2.3]{Grant}.
\end{remark}

We now discuss the relation between this paper and
Sommers's conjecture (\cite[Conjecture 4.2]{Som}).
This conjecture would (as observed in \cite{Som}) imply formulas
for the graded $G$-module decomposition for $R(\widetilde{\co})$.
For orbits that are not even, this conjecture differs what is done
in this paper, because it is based on a resolution built
from a Jacobson-Morozov parabolic, so we assume $\co$ is even.
Let $\{ h', e, f \}$ be an $\mathfrak{sl}_2$ triple containing $e$.
(The element $h'$
corresponds to our $-h$, since our convention is $[h, e] = -2e$.)  Let $\fl'$ denote the centralizer
of $h'$ in $\fg$, let $\fu'$ be the span of the positive eigenspaces for $h'$,
and let $\fp' = \fl' + \fu'$.
Then  \cite[Conjecture 4.2]{Som} is equivalent to the statement that if $\lambda$ is a minimal lift of $\chi_j$
to $P'$, then $R^d(\mcM)_{\chi_j} = H^0(G/P', R^{d + \lambda(h')/2} (\fu') \otimes \C_{\lambda})$,
and the higher cohomology vanishes.  This coincides with the statement in
Proposition \ref{p:conjecture} for even orbits
and the Jacobson-Morozov parabolic, since by Proposition \ref{jlifts}, the lifts in that proposition are exactly the minimal lifts.
(The translation from \cite{Som} uses the fact that the grading on $R(\mcM)$  in \cite{Som} is double
ours, so the index $n$ in $\Gamma^n_V$ corresponds to degree $d = n/2$ in our conventions.
Also, we have replaced cohomology on $G/B$ with cohomology on $G/P'$, since $V$ is a $P$-module,
then $H^i(G/B, V) = H^i(G/P', V)$, by \cite[Proposition II.4.6(b)]{Jant}.)

\bibliographystyle{amsplain}

\end{document}